\documentclass{amsart}

\usepackage{color}
\usepackage[dvips]{graphicx}
\usepackage{amscd}
\usepackage{amsmath}
\usepackage{ascmac}
\usepackage{cases}
\usepackage{amssymb}
\usepackage{amsfonts}
\usepackage{amsthm}
\usepackage[all]{xy}
\usepackage{enumerate}

\theoremstyle{definition}
\newtheorem{thm}{{\bf Theorem}}[section]

\newtheorem{lem}[thm]{{\bf Lemma}}
\newtheorem{prop}[thm]{{\bf Proposition}}

\newtheorem{conj}[thm]{{\bf Conjecture}}
\newtheorem{dfn}[thm]{{\bf Definition}}
\newtheorem{atte}[thm]{{\bf Remark}}

\newcounter{Exami}

\makeatother

\title{On compactifications of character varieties of $n$-punctured projective line }
\author{Arata Komyo}
\address{Department of Mathematics, Graduate School of Science, Kobe University, 1-1 Rokkodai-cho, Nada-ku, Kobe, 657-8501, Japan}
\email{akomyo@math.kobe-u.ac.jp}

\begin{document}

\maketitle

\begin{abstract}
In this paper, we construct compactifications of $SL_2(\mathbb{C})$-character varieties of $n$-punctured projective line and study the boundary divisor of the compactifications.
This study is motivated by the conjecture for the configuration of the boundary divisor, due to C. Simpson.
We verify the conjecture for a few examples. 
\end{abstract} 

\section{Introduction}

Let $C$ be a compact Riemann surface of genus $g$, and let $\{ t_1,\ldots,t_n \}$ be the set of $n$-distinct points on $C$.
For a positive integer $r>0$, denote by $\mathcal{P}_r$ the set of partitions of $r$,
and fix $\boldsymbol{\mu}=(\mu^1,\ldots,\mu^n) \in (\mathcal{P}_r)^n$ where $\mu^i=(\mu^i_1,\ldots, \mu^i_{r_i}) \in \mathcal{P}_r$.
For each partition $\mu^i\in \mathcal{P}_r$, let us fix semisimple conjugacy classes $\mathcal{C}_1,\ldots,\mathcal{C}_n \subset SL_r(\mathbb{C})$ which is generic in the sense of \cite[Definition 2.1.1]{HLR}   
and type $\mu^1,\ldots,\mu^n$, that is, 
the multiplicities of eigenvalues of matrices in $\mathcal{C}_i$ are given by $\mu^i=( \mu^i_1,\mu^i_2,\ldots )$.
We consider a monodoromy $SL_r({\mathbb{C}})$-semisimple representation 
\begin{equation*}
\rho : \pi_1(C\setminus \{ t_1,\ldots,t_n\},*) \longrightarrow SL_r(\mathbb{C})
\end{equation*}
of type $(g,\boldsymbol{\mu})$ which satisfies the condition $\rho (\gamma_i) \in \mathcal{C}_i$ for each $i$ where $\gamma_i$ is a anticlockwise loop around the point $t_i$.
We can define the $SL_r(\mathbb{C})$-\textit{character variety} $\mathcal{R}_{g,\boldsymbol{\mu}}$ of the $n$-punctured compact Riemann surface of genus $g$ by the following categorical quotient
\begin{align*}
\mathcal{R}_{g, \boldsymbol{\mu}}:= \{ (A_1,B_1, &\ldots, A_g,B_g; M_1, \ldots ,M_n) \in SL_r(\mathbb{C})^{2g} \times \mathcal{C}_1 \times \cdots \times \mathcal{C}_n  \\
& \mid (A_1,B_1) \cdots (A_g,B_g)M_1 \cdots M_n =I_r\}/\!/SL_r(\mathbb{C}).
\end{align*}
Here, we set $(A,B)=ABA^{-1}B^{-1}$ and $I_r$ is the identity matrix.
The variety depends on the actual choice of eigenvalues, but for simplicity we drop this choice from the notation.
The categorical quotient $\mathcal{R}_{g,\boldsymbol{\mu}}$ can be considered as a moduli space of 
monodoromy $SL_r({\mathbb{C}})$-semisimple representations of type $(g,\boldsymbol{\mu})$.
The variety $\mathcal{R}_{g,\boldsymbol{\mu}}$, if nonempty, is a nonsingular affine variety of dimension 
\begin{equation*}
d_{g,\boldsymbol{\mu}}:= r^2(2g-2+n) - \sum_{i,j}(\mu^i_j)^2+2-2g.
\end{equation*}
(See \cite{HLR}). 
In the case where $g=0$ and $d_{g,\boldsymbol{\mu}}=2$, $SL_r(\mathbb{C})$-character varieties can be classified into four cases, which can be listed as follows:
\begin{equation}\label{dim2 list}
\begin{aligned}
\boldsymbol{\mu}&=((1,1),(1,1),(1,1),(1,1)) \\
\boldsymbol{\mu}&=((1,1,1),(1,1,1),(1,1,1)) \\
\boldsymbol{\mu}&=((2,2),(1,1,1,1),(1,1,1,1))\\
\boldsymbol{\mu}&=((3,3),(2,2,2),(1,1,1,1,1,1)).
\end{aligned}
\end{equation}
In the first and second types, the $SL_r(\mathbb{C})$-character varieties are known to be an affine cubic surface. (\cite{FK}, \cite{Jim}, \cite{Iwa}, \cite{Lawton}).

The purpose of this paper is to study the configuration of boundary divisor of compactifications of $SL_r(\mathbb{C})$-character varieties. 
This study is motivated by a conjecture due to Simpson \cite{Sim3}, which is explained as follows.
We choose a smooth compactification $\overline{\mathcal{R}_{g,\boldsymbol{\mu}}}$ of $\mathcal{R}_{g,\boldsymbol{\mu}}$ 
such that $D^{{\it B}}_{g,\boldsymbol{\mu}} = \overline{\mathcal{R}_{g,\boldsymbol{\mu}}} \setminus \mathcal{R}_{g,\boldsymbol{\mu}}$ is a divisor with normal crossings.
We call the divisor $D^{{\it B}}_{g,\boldsymbol{\mu}}$ a \textit{boundary divisor} of the compactification $\overline{\mathcal{R}_{g,\boldsymbol{\mu}}}$.
Let $\overline{N}^{{\it B}}_{g, \boldsymbol{\mu}}$ be a small neighborhood of $D^{{\it B}}_{g, \boldsymbol{\mu}}$ in $\overline{\mathcal{R}_{g,\boldsymbol{\mu}}}$, 
and let $N^{{\it B}}_{g, \boldsymbol{\mu}} = \overline{N}^{{\it B}}_{g, \boldsymbol{\mu}}\cap \mathcal{R}_{g, \boldsymbol{\mu}}=\overline{N}^{{\it B}}_{g, \boldsymbol{\mu}} \setminus D^{{\it B}}_{g, \boldsymbol{\mu}}$.
Let $\Delta( D^{{\it B}}_{g,\boldsymbol{\mu}})$ be
a simplicial complex whose $n$-dimensional simplices correspond to the irreducible components of intersections of $k+1$ distinct components of $D^{{\it B}}_{g, \boldsymbol{\mu}}$.
This is called the \textit{boundary complex} or \textit{Stepanov complex} of a compactification of $\mathcal{R}_{g,\boldsymbol{\mu}}$ (see \cite{Ste}, \cite{Thu}, and \cite{Pay}).
\begin{thm}[{\cite{Ste}, \cite{Thu}, and \cite{Pay}}]
\textit{
The homotopy type of boundary complex $\Delta( D^{{\it B}}_{g,\boldsymbol{\mu}})$ is independent of the choice of compactifications.
}
\end{thm}
We have a continuous map, well-defined up to homotopy,
\begin{equation}
N^{{\it B}}_{g, \boldsymbol{\mu}} \longrightarrow \Delta( D^{{\it B}}_{g,\boldsymbol{\mu}}).
\end{equation}
On the other hand, let $\mathcal{M}_{g, \boldsymbol{\mu}}$ be the moduli space of parabolic Higgs bundles, 
which is diffeomorphic to the character variety $\mathcal{R}_{g, \boldsymbol{\mu}}$ via the non-abelian Hodge theory \cite{Sim90}.
In particular, we have $\dim \mathcal{M}_{g, \boldsymbol{\mu}}=d_{g,\boldsymbol{\mu}}$.
We have the Hitchin fibration $\mathcal{M}_{g, \boldsymbol{\mu}} \rightarrow \mathbb{A}^{\frac{d_{g,\boldsymbol{\mu}}}{2}}$.
The moduli space $\mathcal{M}_{g, \boldsymbol{\mu}}$ has a canonical orbifold compactification, where the divisor at infinity is the quotient
\begin{equation*}
D^{{\it Dol}}_{g, \boldsymbol{\mu}} := \mathcal{M}_{g, \boldsymbol{\mu}}^*/\mathbb{C}^*.
\end{equation*}
Here, $\mathcal{M}_{g, \boldsymbol{\mu}}^*$ is the complement of the nilpotent cone.
Let $\overline{N}^{{\it Dol}}_{g, \boldsymbol{\mu}}$ be a small neighborhood of $D^{{\it Dol}}_{g, \boldsymbol{\mu}}$, 
and let $N^{{\it Dol}}_{g, \boldsymbol{\mu}} = \overline{N}^{{\it Dol}}_{g, \boldsymbol{\mu}}\cap \mathcal{R}_{g, \boldsymbol{\mu}}=\overline{N}^{{\it Dol}}_{g, \boldsymbol{\mu}} \setminus D^{{\it Dol}}_{g, \boldsymbol{\mu}}$.
The Hitchin fibration gives us a continuous map to the sphere at infinity in the Hitchin base
\begin{equation}
N^{{\it Dol}}_{g, \boldsymbol{\mu}} \longrightarrow S^{d_{g,\boldsymbol{\mu}} -1}.
\end{equation}
\begin{conj}[\cite{Sim3}]\label{conj1}
\begin{enumerate}
\item \textit{There exists a homotopy-commutative diagram}
\begin{equation*}
\begin{CD}
N^{{\it Dol}}_{g, \boldsymbol{\mu}} @>{\cong}>> N^{{\it B}}_{g, \boldsymbol{\mu}}  \\
@VVV @VVV \\
S^{d_{g,\boldsymbol{\mu}} -1} @>{\cong}>>\Delta( D^{{\it B}}_{g,\boldsymbol{\mu}}).
\end{CD}
\end{equation*}
\item \textit{
In particular,
there exists a non-singular compactification of $\mathcal{R}_{g, \boldsymbol{\mu}}$ such that the boundary complex is a simplicial decomposition of sphere $S^{d_{g,\boldsymbol{\mu}} -1}$.
}
\end{enumerate}
\end{conj}

\begin{atte}[See \cite{Sim3}]
The assertion (1) of Conjecture \ref{conj1} is true in the first case of the list (\ref{dim2 list}).
\end{atte}

The main theorem of this paper is the following
\begin{thm}[Theorem \ref{thm}]\label{thm intro}
\textit{The assertion} (2) {of Conjecture} \ref{conj1} \textit{is true in the following cases}: 
\begin{enumerate}
\item $g=0,r=3,n=3,\boldsymbol{\mu}=((1,1,1),(1,1,1),(1,1,1)),\ d_{g,\boldsymbol{\mu}}=2;$
\item $g=0,r=2,n=5,\boldsymbol{\mu}=((1,1),(1,1),(1,1),(1,1),(1,1))\ d_{g,\boldsymbol{\mu}}=4.$
\end{enumerate}
\end{thm}
For the case $(1)$ of Theorem \ref{thm intro}, the assertion (2) of Conjecture \ref{conj1} can be verified by the classical invariant theory. 
(\cite{FK}, \cite{Jim}, \cite{Iwa}, \cite{Lawton}).
However, it seems that the application of the classical invariant theory is difficult for general cases.
Then, we construct compactifications of $SL_r(\mathbb{C})$-character varieties as follows.
Following \cite{abc}, we can construct a compactification of the \textit{representation variety} \cite{abc}
\begin{align*}
\mathrm{Rep}_{g, \boldsymbol{\mu}}:=  \{ (A_1,B_1, \ldots, A_g,B_g;& M_1, \ldots ,M_n) \in SL_r(\mathbb{C})^{2g} \times \mathcal{C}_1 \times \cdots \times \mathcal{C}_n  \\
& \mid (A_1,B_1) \cdots (A_g,B_g)M_1 \cdots M_n =I_r\}.
\end{align*}
Then, we take the GIT quotient of this compactification of $\mathrm{Rep}_{g, \boldsymbol{\mu}}$, 
which gives a compactification $\overline{\mathcal{R}_{g, \boldsymbol{\mu}}}$ of $\mathcal{R}_{g, \boldsymbol{\mu}}$.
As special cases, we consider the case where $g=0,r=2,n\ge4, \boldsymbol{\mu}=((1,1),\ldots,(1,1))$. 
For $n=4$, we obtain the same result as the classical invariant theory \cite{FK}.
For $n=5$ (i.e., the case $(2)$ of Theorem \ref{thm intro}), $\overline{\mathcal{R}_{g, \boldsymbol{\mu}}}$ has singular points.
A suitable blowing up of $\overline{\mathcal{R}_{g, \boldsymbol{\mu}}}$ shows that the assertion (2) of Conjecture \ref{conj1} holds.
It seems that the configuration of the boundary divisor $D^{{\it B}}_{0,\boldsymbol{\mu}}$ is rather complicated for $n\ge6$.

Conjecture \ref{conj1} is related to the P=W conjecture due to Hausel et al (\cite{CHM}).
First, we consider compact curve cases.
The non-abelian Hodge theory for compact curves 
states that character varieties $\mathcal{R}$ are diffeomorphic to moduli spaces $\mathcal{M}$ of semi-stable Higgs bundles.
Then, we have the induced isomorphism between the rational cohomology groups of $\,\mathcal{R}$ and $\mathcal{M}$.
The P=W conjecture assert that the isomorphism of the rational cohomology groups exchanges 
the weight filtration on the cohomology groups of $\mathcal{R}$ 
with the perverse Leray filtration associated with the Hitchin fibration on the cohomology groups of $\mathcal{M}$.
The P=W conjecture is verified in the case where $r=2$ (\cite{CHM}).
We may extend the conjecture to punctured curve cases.  
On the other hand, there exists a natural isomorphism from the reduced homology of the boundary complex $\Delta( D^{{\it B}}_{g,\boldsymbol{\mu}})$
 to the $2l$-th graded piece of the weight filtration on the cohomology of $\mathcal{R}_{g, \boldsymbol{\mu}}$:
\begin{equation*}
\widetilde{H}_{i-1} ( \Delta( D^{{\it B}}_{g,\boldsymbol{\mu}}), \mathbb{Q}) \cong Gr^W_{2l} H^{2l-i} ( \mathcal{R}_{g, \boldsymbol{\mu}}, \mathbb{Q} ).
\end{equation*}
(For example, see \cite[Theorem 4.4]{Pay}).
By the isomorphism, the assertion (2) of Conjecture \ref{conj1} implies that there exists only $1$-dimensional weight $2d_{g,\boldsymbol{\mu}}$ part 
in the middle degree $d_{g,\boldsymbol{\mu}}$ cohomology of the character variety,
which is also a consequence of the P=W conjecture.

\begin{atte}
The structure groups of character varieties studied in \cite{CHM} are $GL_n(\mathbb{C})$, $PGL_n(\mathbb{C})$ and $SL_n(\mathbb{C})$. 
However, for $g=0$, those character varieties are the same.
\end{atte}

The organization of this paper is as follows.
In Section 2, we give the definition of a $SL_r(\mathbb{C})$-character variety.
In Section 3, we consider the case where $g=0,r=2, n=4$ and $g=0,r=3,n=3$.
In those cases, the character varieties are describe by invariants and a relation of invariants.
We recall that the character varieties are affine cubic surfaces.
In Section 4, we consider the construction of compactifiations of $SL_2(\mathbb{C})$-character varieties of $g=0, \boldsymbol{\mu}=((1,1),\ldots,(1,1))$.
In Section 5 and 6, we describe the boundary divisor of the compactifiations of the cases where $n=4$ and $n=5$.

\renewcommand{\abstractname}{Acknowledgements}
\begin{abstract}
The author would like to thank Professor Kentaro Mitsui, Professor Masa-Hiko Saito and Professor Carlos Simpson, and  for many comments and discussions.
He thanks Professor Masa-Hiko Saito for warm encouragement.
\end{abstract}

\section{Preliminaries}

We fix integers $g,r,n$ with $g\ge 0, r>0,n>0$, and
let $(C,\boldsymbol{t})=(C,t_1,\ldots,t_n)$ be an $n$-pointed compact Riemann surface of genus $g$, 
which consists of a compact Riemann surface $C$ of genus $g$ and a set of $n$-distinct points $\boldsymbol{t}=\{ t_i\}_{1\le i\le n}$ on $C$.
We put $D(\boldsymbol{t})=t_1 +\cdots + t_n$ for each $(C,\boldsymbol{t})=(C,t_1,\ldots,t_n)$. 
We denote by 
\begin{equation}
\Gamma_{C,\boldsymbol{t}} :=\pi_1(C\setminus D(\boldsymbol{t}),*)
\end{equation}
the fundamental group of $C\setminus D(\boldsymbol{t})$ with the base point $*\in C\setminus D(\boldsymbol{t})$. 
The group $\Gamma_{C,\boldsymbol{t}}$ is generated by $(2g+n)$-element $\alpha_1,\ldots,\alpha_g,\beta_1,\ldots,\beta_g,\gamma_1,\ldots,\gamma_n$ with one relation 
\begin{equation*}
(\alpha_1,\beta_1) \cdots (\alpha_g, \beta_g)\gamma_1 \cdots \gamma_n=1.
\end{equation*}
Here, we set $(\alpha,\beta)=\alpha\beta\alpha^{-1}\beta^{-1}$.
The set of generators $\alpha_1,\ldots,\alpha_g,\beta_1,\ldots,\beta_g,\gamma_1,\ldots,\gamma_n$ is called \textit{canonical generators} of $\Gamma_{C,\boldsymbol{t}}$.
\begin{dfn}
An $SL_r(\mathbb{C})$-representation of the fundamental group $\Gamma_{C,\boldsymbol{t}}$ is a group homomorphism 
\begin{equation}
\rho : \Gamma_{C,\boldsymbol{t}} \longrightarrow SL_r(\mathbb{C}). 
\end{equation}
\end{dfn}
Let $\mathrm{Hom}(\Gamma_{C,\boldsymbol{t}},SL_r(\mathbb{C}))$ be the set of all $SL_r(\mathbb{C})$-representations of $\Gamma_{C,\boldsymbol{t}}$.
If we fix a set of canonical generators of $\Gamma_{C,\boldsymbol{t}}$, we have the identification
\begin{equation*}
\mathrm{Hom}(\Gamma_{C,\boldsymbol{t}},SL_r(\mathbb{C}))\xrightarrow{\simeq}SL_r(\mathbb{C})^{2g+n-1}.
\end{equation*}

\begin{dfn}
Two $SL_r(\mathbb{C})$-representations $\rho_1$ and $\rho_2$ are isomorphic to each other, if and only if there exists a matrix $P\in SL_r(\mathbb{C})$ such that 
\begin{equation*}
\rho_2 (\gamma)=P^{-1} \cdot \rho_1(\gamma) \cdot P \mbox{ for all } \gamma \in \Gamma_{C,\boldsymbol{t}}.
\end{equation*}
\end{dfn}
Let $R^r_{(g,n-1)}$ denote the affine coordinate ring of $SL_r(\mathbb{C})^{2g+n-1}$.
We consider the simultaneous action of $SL_r(\mathbb{C})$ on $SL_r(\mathbb{C})^{2g+n-1}$ as
\begin{align*}
P \curvearrowright &(A_1,\ldots,A_g,B_1,\ldots,B_g;M_1,\ldots,M_{n-1}) \\
&\mapsto (P^{-1}A_1P,\ldots , P^{-1}A_gP,P^{-1}B_1P,\ldots , P^{-1}B_gP;P^{-1}M_1P,\ldots , P^{-1}M_{n-1}P).
\end{align*}
The invariant ring $(R^r_{(g,n-1)})^{Ad(SL_r(\mathbb{C}))}$ is finitely generated.
For any $(C,\boldsymbol{t} )$, there exists the universal categorical quotient map 
\begin{equation*}
\Phi^r_{(C,\boldsymbol{t})} : \mathrm{Hom}(\Gamma_{C,\boldsymbol{t}},SL_r(\mathbb{C})) \cong SL_r(\mathbb{C})^{2g+n-1} \rightarrow 
\mathcal{R}^r_{(C,\boldsymbol{t})} =SL_r(\mathbb{C})^{2g+n-1} /\!/ SL_r(\mathbb{C})
\end{equation*}
where 
\begin{equation*}
\mathcal{R}^r_{(C,\boldsymbol{t})} =\mathrm{Spec}[(R^r_{(g,n-1)})^{Ad(SL_r(\mathbb{C}))}].
\end{equation*}
 The following lemme is due to Simpson.
\begin{lem}[{\cite[Proposition 6.1]{Sim2}}]
\textit{
The closed points of $\mathcal{R}^r_{(C,\boldsymbol{t})}$ represent the Jordan equivalence classes of $SL_r(\mathbb{C})$-representations of $\Gamma_{C,\boldsymbol{t}}$.
}
\end{lem}

Let us set 
\begin{equation*}
\mathcal{A}^{(n)}_r:=\left\{ \boldsymbol{a}=(a^{(i)}_j)^{1\le i \le n}_{1\le j \le r-1} \in \mathbb{C}^{nr-n} \right\}.
\end{equation*}
For $\boldsymbol{a}=(a^{(i)}_j)\in \mathcal{A}^{(n)}_r$, we set
\begin{equation*}
\chi_{i}(s):=s^r+a^{(i)}_{r-1}s^{r-1}+\cdots+a^{(i)}_1s+(-1)^r,\ (i=1,\ldots, n).
\end{equation*}
Moreover, we define the morphism
\begin{equation*}
\phi^r_{(C,\boldsymbol{t})} :  \mathcal{R}^r_{(C,\boldsymbol{t})} \rightarrow \mathcal{A}^{(n)}_r
\end{equation*}
by the relation 
\begin{equation*}
\mathrm{det}(sI_r-\rho (\gamma_i ))=\chi_{i}(s)
\end{equation*}
where $[\rho]\in \mathcal{R}^r_{(C,\boldsymbol{t})}$ and $\gamma_i$ is a anticlockwise loop around the point $t_i$.
The fiber of $\phi^r_{(C,\boldsymbol{t})}$ at $\boldsymbol{a}\in \mathcal{A}^{(n)}_r$ is given by the affine subscheme of $\mathcal{R}^r_{(C,\boldsymbol{t})}$:
\begin{align*}
\mathcal{R}^r_{(C,\boldsymbol{t}),\boldsymbol{a}}:=&\ (\phi^r_{(C,\boldsymbol{t})})^{-1}(\boldsymbol{a}) \\
=&\ \{[\rho]\in \mathcal{R}^r_{(C,\boldsymbol{t})}\mid \mathrm{det}(sI_r-\rho(\gamma_i))=\chi_{i}(s), 1 \le i \le n\}.
\end{align*}
For $\boldsymbol{a} \in \mathcal{A}^{(n)}_r$, 
let $\mu^i=(\mu^i_1,\mu^i_2,\ldots)$ be the partition of $r$ which implies the multiplicity of the solutions of the equation $\chi_i(s)=0$.
Put $\boldsymbol{\mu}=(\mu^1,\ldots,\mu^n)$, called the \textit{multiplicity of $\boldsymbol{a} \in \mathcal{A}^{(n)}_r$}.
Moreover, we define the subvariety 
\begin{equation*}
\mathcal{A}^{(n)}_{r, \boldsymbol{\mu}}:= \left\{ \boldsymbol{a}=(a^{(i)}_j)^{1\le i \le n}_{1\le j \le r-1} \in \mathbb{C}^{nr-n}\ 
\middle|\ \mbox{the multiplicity of $\boldsymbol{a}$ is $\boldsymbol{\mu}$} \right\} 
\subset \mathcal{A}^{(n)}_r.
\end{equation*}
\begin{dfn}
We fix a $k$-tuple $\boldsymbol{\mu}$ of partitions of $r$. 
Let $\boldsymbol{a}$ be a element of $\mathcal{A}^{(n)}_{r, \boldsymbol{\mu}}$.
Then, we define
\begin{align*}
\mathcal{R}^{r,\textit{s}}_{(C,\boldsymbol{t}),\boldsymbol{\mu},\boldsymbol{a}}:=
&\ \{[\rho]\in \mathcal{R}^r_{(C,\boldsymbol{t})}\mid \mathrm{det}(sI_r-\rho(\gamma_i))=\chi_{i}(s), \rho(\gamma_i): \mbox{diagonalizable}, 1 \le i \le n\}\\
=&\ \{ (A_1,B_1, \ldots, A_g,B_g; M_1, \ldots ,M_n) \in SL_r(\mathbb{C})^{2g} \times \mathcal{C}_1 \times \cdots \times \mathcal{C}_n  \\
&\qquad \qquad \quad \mid (A_1,B_1) \cdots (A_g,B_g)M_1 \cdots M_n =I_r\}/\!/SL_r(\mathbb{C})
\end{align*}
where $\mathcal{C}_i=\{  M \in SL_r(\mathbb{C}) \mid \mathrm{det}(sI_r-M)=\chi_{\boldsymbol{a}^{(i)}}(s),\ M:\mbox{diagnalizable} \}$.
In Section 1, we denoted by $\mathcal{R}_{g,\boldsymbol{\mu}}$ the variety instead of $\mathcal{R}^{r,\textit{s}}_{(C,\boldsymbol{t}),\boldsymbol{\mu},\boldsymbol{a}}$, for simplicity.
The affine subvariety $\mathcal{R}^{r,\textit{s}}_{(C,\boldsymbol{t}),\boldsymbol{\mu},\boldsymbol{a}}$ is called a 
$SL_r(\mathbb{C})$-\textit{character variety} of the $n$-punctured compact Riemann surface of genus $g$.
In particular, we denote by $\mathcal{R}^{r}_{n,\boldsymbol{a}}$ this variety in the case where $g=0,\boldsymbol{\mu}=((1,\ldots,1),\ldots,(1,\ldots,1))$.
\end{dfn}

If we take a generic $\boldsymbol{a} \in \mathcal{A}^{(n)}_{r, \boldsymbol{\mu}}$, 
the affine algebraic variety $\mathcal{R}^{r,\textit{s}}_{(C,\boldsymbol{t}),\boldsymbol{\mu},\boldsymbol{a}}$ is a non-singular irreducible variety of dimension
\begin{equation*}
d_{g,\boldsymbol{\mu}}:= r^2(2g-2+n) - \sum_{i,j}(\mu^i_j)^2+2-2g,
\end{equation*}
and has a holomorphic symplectic structure,  if nonempty. (See \cite{HLR},\cite{IIS}).
In particular, for $g=0,\boldsymbol{\mu}=((1,\ldots,1),\ldots,(1,\ldots,1))$, the dimension of  $\mathcal{R}^r_{n,\boldsymbol{a}}$ is
\begin{equation*}
d_{0,((1,1),\ldots,(1,1))}= 2 n-6.
\end{equation*}

\section{Invariant ring }

Now, we recall the explicit description of the invariant ring $(R^r_{(g,n-1)})^{Ad(SL_r(\mathbb{C}))}$ for the two cases $g=0,r=2,n=4$ and $g=0,r=3,n=3$.
The following proposition follows from the fundamental theorem for matrix invariants. (See \cite{For} or \cite{Pro}).
\begin{prop}
\textit{
\begin{equation*}
(R^r_{(0,n-1)})^{Ad(SL_r(\mathbb{C}))}=\mathbb{C}[\mathrm{Tr}(M_{i_1}M_{i_2}\cdots M_{i_{k}})\mid 1 \le i_1,\ldots,i_k \le n-1].
\end{equation*}
In particular, for $r=2$, the elements $\mathrm{Tr}(M_{i_1}M_{i_2}\cdots M_{i_{k}})$ of degree $k\le 3$ generate the invariant ring, that is,
\begin{equation*}
(R^2_{(0,n-1)})^{Ad(SL_2(\mathbb{C}))}=\mathbb{C}[\mathrm{Tr}(M_i),\mathrm{Tr}(M_iM_j),\mathrm{Tr}(M_iM_jM_k) \mid 1 \le i,j,k \le n-1].
\end{equation*}
}
\end{prop}

First, we consider the case where $g=0,r=2,n=4$.
Let $(i,j,k)$ be a cyclic permutation of $(1,2,3)$. 
Then, the invariant ring $(R^2_{(0,3)})^{Ad(SL_2(\mathrm{C}))}$ is generated by
\begin{equation}
x_i:=\mathrm{Tr}(M_k M_j)\  (i=1,2,3), a_i:=\mathrm{Tr}(M_i)\ (i=1,2,3), a_4:=\mathrm{Tr}(M_3 M_2 M_1).
\end{equation}
The following proposition is due to Frike-Klein, Jimbo, and Iwasaki, (\cite{FK}, \cite{Jim}, \cite{Iwa}).

\begin{prop}
\textit{
The invariant ring $(R^2_{(0,3)})^{Ad(SL_2(\mathrm{C})))}$ is generated by seven elements $x_1,x_2,x_3,a_1,a_2,a_3,a_4$ and there exists a relation
\begin{equation*}
f_{\boldsymbol{a}}(x):= x_1 x_2 x_3+x_1^2+x_2^2+x_3^2-\theta_1(\boldsymbol{a})x_1-\theta_2(\boldsymbol{a})x_2-\theta_3(\boldsymbol{a})x_3+\theta_4(\boldsymbol{a})=0
\end{equation*}
where
\begin{align*}
&\theta_i (\boldsymbol{a})=a_i a_4+a_j a_k\quad  (i,j,k),\\
&\theta_4 (\boldsymbol{a})=a_1 a_2 a_3 a_4 +a_1^2+a_2^2+a_3^2+a_4^2-4.
\end{align*}
Therefore, we have an isomorphism
\begin{equation*}
(R^2_{(0,3)})^{Ad(SL_2(\mathbb{C})))} \cong \mathbb{C}[x_1,x_2,x_3,a_1,a_2,a_3,a_4]/(f_{\boldsymbol{a}}(x)).
\end{equation*}
}
\end{prop}

We have the surjective morphism
\begin{equation*}
\phi^2_{(\mathbb{P}^1,0,1,t,\infty)} :  \mathcal{R}^2_{(\mathbb{P}^1,0,1,t,\infty)}=\mathrm{Spec}[(R^2_{(0,3)})^{Ad(SL_2(\mathbb{C})))}] 
\rightarrow \mathcal{A}^{(4)}_2=\mathrm{Spec}[\mathbb{C}[a_1,a_2,a_3,a_4]]
\end{equation*}
where $t$ is a point of $\mathbb{P}^1$ such that $t\neq 0,1,\infty$.
The fiber at $\boldsymbol{a}\in \mathcal{A}^{(4)}_2$, such that the type of the multiplicities of eigenvalues is $((1,1),(1,1),(1,1),(1,1))$, is an affine cubic hypersurface in $\mathbb{C}^3$.
Hence, the $SL_2(\mathbb{C})$-character variety of the $4$-punctured projective line is an affine cubic hypersurface
\begin{equation*}
\mathcal{R}_{4,\boldsymbol{a}} \cong \{ (x_1,x_2,x_3) \in \mathbb{C}^3 \mid f_{\boldsymbol{a}}(x)=0 \}.
\end{equation*}
The affine cubic hypersurface is called a \textit{Fricke-Klein cubic surface}.

We consider the natural compactification $\mathbb{C}^3\hookrightarrow \mathbb{P}^3$ as follows.
Set $x_1=X/W,x_2=Y/W,x_3=Z/W$.
Then, we obtain the following homogeneous polynomial
\begin{equation*}
X Y Z+X^2 W+Y^2 W+Z^2 W-\theta_1(a) X W^2-\theta_2(a) Y W^2-\theta_3(a) Z W^2+\theta_4(a) W^3=0.
\end{equation*}
Substitute $W=0$ to this equation.
Then, we obtain the equation $XYZ=0$.
Hence, the boundary divisor of the natural compactification of $\mathcal{R}_{4,\boldsymbol{a}}$ consists of three lines.
The boundary complex is shown in Figure 1.
\begin{figure}
\unitlength 0.1in
\begin{picture}( 26.3000,  9.2000)(  9.5000,-16.9000)
%
{\color[named]{Black}{%
\special{pn 8}%
\special{pa 1624 770}%
\special{pa 1094 1650}%
\special{fp}%
\special{pa 1070 1438}%
\special{pa 2062 1438}%
\special{fp}%
\special{pa 1312 778}%
\special{pa 1960 1650}%
\special{fp}%
}}%
\put(9.5000,-12.2000){\makebox(0,0)[lb]{$X=0$}}%
\put(13.5000,-15.6000){\makebox(0,0)[lb]{$Y=0$}}%
\put(16.8900,-12.2700){\makebox(0,0)[lb]{$Z=0$}}%
%
{\color[named]{Black}{%
\special{pn 8}%
\special{pa 2176 1186}%
\special{pa 2640 1186}%
\special{fp}%
\special{sh 1}%
\special{pa 2640 1186}%
\special{pa 2574 1166}%
\special{pa 2588 1186}%
\special{pa 2574 1206}%
\special{pa 2640 1186}%
\special{fp}%
}}%
%
{\color[named]{Black}{%
\special{pn 8}%
\special{pa 2940 1132}%
\special{pa 2940 1132}%
\special{fp}%
}}%
%
{\color[named]{Black}{%
\special{pn 4}%
\special{sh 1}%
\special{ar 2952 1064 16 16 0  6.28318530717959E+0000}%
\special{sh 1}%
\special{ar 3548 1064 16 16 0  6.28318530717959E+0000}%
\special{sh 1}%
\special{ar 3266 1540 16 16 0  6.28318530717959E+0000}%
\special{sh 1}%
\special{ar 3266 1540 16 16 0  6.28318530717959E+0000}%
}}%
%
{\color[named]{Black}{%
\special{pn 8}%
\special{pa 2958 1064}%
\special{pa 3548 1064}%
\special{fp}%
\special{pa 3548 1064}%
\special{pa 3266 1534}%
\special{fp}%
\special{pa 3266 1534}%
\special{pa 2958 1078}%
\special{fp}%
}}%
%
{\color[named]{Black}{%
\special{pn 0}%
\special{sh 1.000}%
\special{ia 2952 1064 34 38  0.0000000  6.2831853}%
}}%
{\color[named]{Black}{%
\special{pn 4}%
\special{pn 4}%
\special{ar 2952 1064 34 38  0.0000000  6.2831853}%
}}%
%
{\color[named]{Black}{%
\special{pn 0}%
\special{sh 1.000}%
\special{ia 3266 1540 32 38  0.0000000  6.2831853}%
}}%
{\color[named]{Black}{%
\special{pn 4}%
\special{pn 4}%
\special{ar 3266 1540 32 38  0.0000000  6.2831853}%
}}%
%
{\color[named]{Black}{%
\special{pn 0}%
\special{sh 1.000}%
\special{ia 3548 1070 32 38  0.0000000  6.2831853}%
}}%
{\color[named]{Black}{%
\special{pn 4}%
\special{pn 4}%
\special{ar 3548 1070 32 38  0.0000000  6.2831853}%
}}%
\put(11.6600,-18.2000){\makebox(0,0)[lb]{$XYZ=0$}}%
\put(21.7000,-10.0000){\makebox(0,0)[lb]{boundary}}%
\put(21.9000,-11.4000){\makebox(0,0)[lb]{complex}}%
\end{picture}%
\caption{} 
\end{figure}
The boundary complex is a simplicial decomposition of $S^1$.

Next, we consider the case where $g=0,r=3,n=3$. 
We describe generators and defining relations for the invariant ring $(R^{3}_{(0,2)})^{Ad(SL_3(\mathbb{C}))}$.
The following proposition is due to Lawton \cite{Lawton}. 
\begin{prop}
\textit{
The invariant ring $(R^{3}_{(0,2)})^{Ad(SL_3(\mathbb{C}))}$ is generated by
\begin{align*}
a_1 &:= \mathrm{Tr}(M_1) &   &a_2 := \mathrm{Tr}(M_1^{-1})\\
b_1 &:= \mathrm{Tr}(M_2)  & &b_2 := \mathrm{Tr}(M_2^{-1}) \\
c_1 &:= \mathrm{Tr}(M_1^{-1}M_2^{-1})=\mathrm{Tr}(M_3) & &c_2 := \mathrm{Tr}(M_1 M_2)=\mathrm{Tr}(M_3^{-1}) \\
x_1 &:= \mathrm{Tr}(M_1 M_2^{-1}) & &x_2 := \mathrm{Tr}(M_1^{-1}M_2) \\
x_3 &:= \mathrm{Tr}(M_1 M_2 M_1^{-1}M_2^{-1}),
\end{align*}
and there exists a relation
\begin{equation*}
x_3^2-f x_3+g=0
\end{equation*}
where $f,g$ are polynomials of  $x_1,x_2$ over $\mathbb{C}[a_1,a_2,b_1,b_2,c_1,c_2]$, more precisely,
\begin{align*}
f &= x_1x_2-a_2b_1x_1-a_1b_2x_2+(\mbox{constant terms in $x_1,x_2$}) \\
g &= x_1^3+x_2^3 +(\mbox{terms that order is at most $2$ in $x_1,x_2$} ).
\end{align*}
}
\end{prop}
We consider the subring $A^{(3)}_3=\mathbb{C}[a_1,a_2,b_1,b_2,c_1,c_2]$ of $(R^{3}_{(0,2)})^{Ad(SL_3(\mathbb{C}))}$.
We have a natural morphism
\begin{equation*}
\phi^{3}_{(\mathbb{P}^1,0,1,\infty)} :  \mathcal{R}^{3}_{(\mathbb{P}^1,0,1,\infty)}=\mathrm{Spec}[(R^{3}_{(0,2)})^{Ad(SL_3(\mathbb{C}))}]  \rightarrow \mathcal{A}^{(3)}_3=\mathrm{Spec}[A^{(3)}_3].
\end{equation*}
The fiber at $\boldsymbol{a}\in \mathcal{A}^{(3)}_3$, such that the type of the multiplicities of eigenvalues is $((1,1,1),(1,1,1),(1,1,1))$, is an affine cubic hypersurface in $\mathbb{C}^3$.
Hence, the $SL_3(\mathbb{C})$-character variety of the $3$-punctured projective line is an affine cubic hypersurface
\begin{equation*}
\mathcal{R}^3_{3,\boldsymbol{a}} \cong \{ (x_1,x_2,x_3) \in \mathbb{C}^3 \mid x_3^2-f x_3+g=0 \}.
\end{equation*}

We consider the compactification $\mathbb{C}^3\hookrightarrow \mathbb{P}^3$ as follows.
Set $x_1=X/W,x_2=Y/W,x_3=Z/W$.
Then, we obtain the following homogeneous polynomial
\begin{equation*}
X^3+Y^3-X Y Z +(\mbox{term containing }W )=0.
\end{equation*}
We substitute $W=0$ to this equation. 
Then, we obtain the equation $X^3+Y^3-X Y Z=0$.
This equation defines a plane cubic curve having a node.
The boundary complex is shown in Figure 2.
\begin{figure}
\input{picture2.tex}
\caption{} 
\end{figure}
The boundary complex is a simplicial decomposition of $S^1$.

\section{ A compactification of the character variety}

We construct a compactification of the $SL_2(\mathbb{C})$-character variety $\mathcal{R}_{n,\boldsymbol{k}}$ 
($\boldsymbol{k}$ of the $n$-punctured projective line is date of coefficient of characteristic polynomials)
by means of the geometric invariant theory for a compactification of the following variety
\begin{dfn}
We put
\begin{equation}
\begin{aligned}
\mathrm{Rep}_{n,\boldsymbol{k}}:=&\ \{ (M_1,\ldots,M_{n-1} ) \in \mathcal{C}_1 \times \cdots \times \mathcal{C}_{n-1}|M_{n-1}^{-1}\cdots M_1^{-1} \in \mathcal{C}_n \} \\
=&\ \{ (M_1,\ldots,M_{n-1} ) \in \mathcal{C}_1 \times \cdots \times \mathcal{C}_{n-1}|\mathrm{Tr}(M_{n-1}^{-1}\cdots M_1^{-1}) =k_n \}
\end{aligned}
\end{equation}
where $\mathcal{C}_i=\{  M \in SL_2(\mathbb{C}) \mid \mathrm{Tr}(M)=k_i  \}$ and $\boldsymbol{k}=(k_1,\ldots,k_n) \in \mathbb{C}^n$.
The affine variety $\mathrm{Rep}_{n,\boldsymbol{k}}$ is said to the $SL_2(\mathbb{C})$-\textit{representation variety} of the $n$-punctured line.
\end{dfn}
We will introduce a compactification of the representation variety due to Benjamin \cite{abc}. 
First, we consider a construction of a compactification of the algebraic group $SL_2(\mathbb{C})$.
We pick an embedding $\alpha\colon SL_2(\mathbb{C}) \hookrightarrow PGL_3(\mathbb{C})$.
Such an embedding always exists: we consider the natural embedding $SL_2(\mathbb{C}) \rightarrow GL_2(\mathbb{C})$ and
we take the composition of the embedding and the map
$GL_2(\mathbb{C}) \xrightarrow[]{\xi} GL_3(\mathbb{C}) \rightarrow PGL_3(\mathbb{C})$ where
\begin{equation*}
\xi (A)=
\arraycolsep5pt
\left(
\begin{array}{@{\,}c|c@{\,}}
A&\\
\hline 
& 1 
\end{array}
\right)
\end{equation*}
and the second arrow is the canonical projection.
We regard $PGL_3(\mathbb{C})$ as an open subvariety of $\mathbb{P}(M_3(\mathbb{C}))$, 
and define the compactification $\overline{SL_2(\mathbb{C})}$ of $SL_2(\mathbb{C})$ as the closure of $\alpha(SL_2(\mathbb{C}))$ in $\mathbb{P}(M_3(\mathbb{C}))$, that is,
\begin{equation*}
\overline{SL_2(\mathbb{C})}=\left\{ \arraycolsep5pt\left(
\begin{array}{@{\,}cc|c@{\,}}
a&b&\\
c&d& \\
\hline 
&& e 
\end{array}
\right) \in \mathbb{P}(M_3(\mathbb{C}))\ \middle| \  ad-bc=e^2  \right\}.
\end{equation*}
Then, we obtain a compactification of the semisimple conjugacy class $\mathcal{C}_i$, denoted by $\overline{\mathcal{C}_i}$, that is,
\begin{equation*}
\overline{\mathcal{C}_i}=\left\{ \arraycolsep5pt\left(
\begin{array}{@{\,}cc|c@{\,}}
a&b&\\
c&d& \\
\hline 
&& e 
\end{array}
\right) \in \mathbb{P}(M_3(\mathbb{C}))\  \middle| \  ad-bc=e^2, \ a+d=k_i e  \right\}.
\end{equation*}
We can define a compactification of the representation variety.
\begin{dfn}
We put
\begin{equation}
\overline{\mathrm{Rep}_{n,\boldsymbol{k}}}:=
\{ (M_1,\ldots,M_{n-1} ) \in \overline{\mathcal{C}}_1 \times \cdots \times \overline{\mathcal{C}}_{n-1} \mid \mathrm{Tr} (A_{1}\cdots A_{n-1})= k_ne_1\cdots e_{n-1} \}
\end{equation}
where 
\begin{equation*}
M_1=
\arraycolsep5pt
\left(
\begin{array}{@{\,}c|c@{\,}}
A_1&\\
\hline 
& e_1 
\end{array}
\right),
\ldots,
M_{n-1}=
\arraycolsep5pt
\left(
\begin{array}{@{\,}c|c@{\,}}
A_{n-1}&\\
\hline 
& e_{n-1} 
\end{array}
\right) .
\end{equation*}
\end{dfn}
\begin{atte}
In general, for $X\in \overline{SL_2(\mathbb{C})}$, there is no inverse.
Since 
\begin{equation*}
\mathrm{Tr}(A_{n-1}^{-1}\cdots A_1^{-1})=\mathrm{Tr}(A_{1}\cdots A_{n-1})
\end{equation*}
 for $\forall A_i\in SL_2(\mathbb{C})$, 
we use the condition $ \mathrm{Tr}(A_{1}\cdots A_{n-1})=k_n $, instead of $\mathrm{Tr}(A_{n-1}^{-1}\cdots A_1^{-1})=k_n$.
\end{atte}

We have the following action of $SL_2(\mathbb{C})$ on $\overline{\mathrm{Rep}_{n,\boldsymbol{k}}}$,
which is compatible with the simultaneous action of $SL_2(\mathbb{C})$ on $\mathrm{Rep}_{n,\boldsymbol{k}}$
\begin{equation}\label{action for proj}
\begin{aligned}
P \curvearrowright &\left(
\arraycolsep5pt
\left(
\begin{array}{@{\,}c|c@{\,}}
A_1&\\
\hline 
& e_1 
\end{array}
\right),
\ldots,
\arraycolsep5pt
\left(
\begin{array}{@{\,}c|c@{\,}}
A_{n-1}&\\
\hline 
& e_{n-1}
\end{array}
\right) \right) \\
&\longmapsto
\left(
\arraycolsep5pt
\left(
\begin{array}{@{\,}c|c@{\,}}
PA_1P^{-1}&\\
\hline 
& e_1 
\end{array}
\right), 
\ldots,
\arraycolsep5pt
\left(
\begin{array}{@{\,}c|c@{\,}}
PA_{n-1}P^{-1}&\\
\hline 
& e_{n-1}
\end{array}
\right)
\right).
\end{aligned}
\end{equation}

We regard $\overline{\mathrm{Rep}_{n,\boldsymbol{k}}} \subset \overline{\mathcal{C}_1} \times \cdots \times \overline{\mathcal{C}_{n-1}}$ 
as the closed subset in $\mathbb{P}^4 \times \cdots  \times \mathbb{P}^4$. 
Then, we obtain an embedding in the projective space by the Segre embedding.
Let $L$ be an ample line bundle associated with this embedding,
that is, 
\begin{equation*}
L= \bigotimes_{i=1}^{n-1} p_i^* (\mathcal{O}_{\mathbb{P}^4}(1))
\end{equation*}
where $p_i \colon \overline{\mathrm{Rep}_{n,\boldsymbol{k}}} \rightarrow \mathbb{P}^4$ is the $i$-th projection.
Then, $L$ admits the $SL_2(\mathbb{C})$-linearization with respect to the action.


For $x=(M_1,\ldots,M_{n-1}) \in \overline{\mathrm{Rep}_{n,\boldsymbol{k}}}$,
we put 
\begin{equation*}
I^{\textit{nil}} :=\{ i \in \{1,\ldots,n-1\} \mid M_i \mbox{ is nilpotent i.e. $e_i=0$ } \}.
\end{equation*}
If $I^{\textit{nil}}$ is not empty, we decompose 
\begin{equation}\label{decomp index set}
I^{\textit{nil}} = I_1^{\textit{nil}}\cup \cdots \cup I_k^{\textit{nil}}
\end{equation}
where the index set $I_l^{\textit{nil}} \subset I^{\textit{nil}}\ (1\le l\le k)$ consists of indexes of same matrices, 
that is, matrices indexed by elements of  $I^{\textit{nil}}_l$ are same each other 
and two matrices which respectively have indexes in $I^{\textit{nil}}_{l}$ and $I^{\textit{nil}}_{l'}$ where $l\neq l'$ are not equal.
Let $\sharp I_l^{\textit{nil}}$ be the cardinality of $I_l^{\textit{nil}}$,
and let $m_1$ be a maximum value in $\sharp I_1^{\textit{nil}}, \ldots,\sharp I_k^{\textit{nil}}$.  
We put
\begin{equation*}
J_l :=\{ j \in \{ 1, \ldots ,n-1\} \mid M_j \mbox{ is not nilpotent, } M_j*M_i = M_i*M_j=M_i,  i \in I_l^{\textit{nil}} \}.  
\end{equation*}
Here, we define the product $*$ as 
\begin{equation*}
M *
M':= 
\arraycolsep5pt
\left(
\begin{array}{@{\,}c|c@{\,}}
AA'&\\
\hline 
& e 
\end{array}
\right) \in \mathbb{P}M_{3}(\mathbb{C})\quad \mbox{for }
M:=\arraycolsep5pt
\left(
\begin{array}{@{\,}c|c@{\,}}
A&\\
\hline 
& e 
\end{array}
\right)
\mbox{ and }
M':=\arraycolsep5pt
\left(
\begin{array}{@{\,}c|c@{\,}}
A'&\\
\hline 
& e' 
\end{array}
\right).
\end{equation*}
Note that the product $*$ is well-defined in the case where $M$ (resp. $M'$) is nilpotent and $M'$ (resp. $M$) is not nilpotent 
where $M \in \overline{\mathcal{C}}$ and $M' \in \overline{\mathcal{C}'}$.
Let $m_2$ be a maximum value in $\{ \sharp J_l  \mid l  \mbox{ is satisfied } \sharp I^{\textit{nil}}_l=m_1 ,1\le l\le k \}$.
If $I^{\textit{nil}}$ is empty, then we put $m_1=m_2=0$.

\begin{atte}
Let $(M_1,\ldots,M_{n-1}) \in \overline{\mathrm{Rep}_{n,\boldsymbol{k}}}$.
Suppose that $i \in I^{\textit{nil}}$.
We normalize the nilpotent matrix $M_i$:
\begin{equation}\label{nilpotent matrix}
M_i=
\arraycolsep5pt
\left(
\begin{array}{@{\,}cc|c@{\,}}
0&1&\\
0&0&\\
\hline 
&& 0
\end{array}
\right).
\end{equation}
For a matrix $M_j$ ($j\neq i$),
the condition which, by this transformation, the matrix $M_j$ is transformed to the following form 
\begin{equation*}
\begin{pmatrix}
a_j&b_j\\
0&d_i
\end{pmatrix}
\end{equation*}
is equivalent to the condition $M_j*M_i = M_i*M_j=M_i$.
\end{atte}

\begin{prop}\label{stability condition}
\textit{
The point $x=(M_1,\ldots,M_{n-1})$ is semi-stable (resp. stable)  point if and only if $x$ is satisfied the following condition,
\begin{equation}
 n-1 \ge 2 m_1 + m_2    \quad (\mbox{resp.}\ >\ ). 
\end{equation}
}
\end{prop}

\begin{proof}
For any integer $r>0$, let $\lambda_r$ be the 1-parameter subgroup (1-PS) of $SL_2(\mathbb{C})$ given by
\begin{equation}\label{1-PS def}
\lambda_r \colon t  \longmapsto \begin{pmatrix}
t^r &0\\
0&t^{-r}
\end{pmatrix}, t \in\mathbb{C}^{\times}.
\end{equation}
The matrix $\lambda_r(t)$ acts on $\overline{\mathrm{Rep}_{n,\boldsymbol{k}}}$ as follows.
\begin{align*}
\begin{pmatrix}
t^r&0\\
0&t^{-r}
\end{pmatrix}
\curvearrowright
&\left(
\arraycolsep5pt
\left(
\begin{array}{@{\,}cc|c@{\,}}
a_1&b_1&\\
c_1&d_1&\\
\hline 
&& e_1
\end{array}
\right),
\ldots,
\arraycolsep5pt
\left(
\begin{array}{@{\,}cc|c@{\,}}
a_{n-1} & b_{n-1} &\\
 c_{n-1} & d_{n-1} &\\
\hline 
&& e_{n-1}
\end{array}
\right)
\right)\\
&\longmapsto
\left(
\arraycolsep5pt
\left(
\begin{array}{@{\,}cc|c@{\,}}
a_1&t^{2r} b_1&\\
t^{-2r} c_1&d_1&\\
\hline 
&& e_1
\end{array}
\right),
\ldots,
\arraycolsep5pt
\left(
\begin{array}{@{\,}cc|c@{\,}}
a_{n-1} & t^{2r} b_{n-1} &\\
t^{-2r} c_{n-1} & d_{n-1} &\\
\hline 
&& e_{n-1}
\end{array}
\right)
\right).
\end{align*}

We put $n':=5^{n-1}$.
Let $\mathbb{A}^{n'}$ be the affine cone over the projective space $\mathbb{P}^{n'-1}$ which is the target space of the Segre embedding.
We take a base change of the affine cone $\mathbb{A}^{n'}$ via $\overline{\mathrm{Rep}_{n,\boldsymbol{k}}} \hookrightarrow \mathbb{P}^{n'-1}$, 
denoted by the same notation $\mathbb{A}^{n'}$. 
Let $x^*=(M_1^*,\ldots, M_{n-1}^*)$ be the closed point of $\mathbb{A}^{n'}$ lying over $x \in \overline{\mathrm{Rep}_{n,\boldsymbol{k}}}$, that is,
$x^* \neq 0$ and $x^*$ projects to $x$.
The action (\ref{action for proj}) and the linearization $L$ define a linear action of $SL_2(\mathbb{C})$ on $\mathbb{A}^{n'}$.
In particular, the matrix $\lambda_r(t)$ acts on $\mathbb{A}^{n'}$ as follows.
For each $i=1,\ldots,n-1$, let $e_1^{(i)},\ldots,e_5^{(i)}$ be a basis of $\mathbb{A}^{5}$ such that the matrix 
\begin{equation*}
M_i^* =\left(
\begin{array}{@{\,}cc|c@{\,}}
a_i&b_i&\\
c_i&d_i&\\
\hline 
&& e_i
\end{array}
\right) 
\end{equation*}
 is describe by
\begin{equation*}
M_i^* = a_i e_1^{(i)}+ b_i e_2^{(i)}+c_i e_3^{(i)}+d_i e_4^{(i)}+e_i e_5^{(i)}.
\end{equation*}
Let $e_{i_1,\ldots,i_{n-1}}$ be the base
$ e^{(1)}_{i_1}\otimes \cdots \otimes e^{(n-1)}_{i_{n-1}}$ of $\mathbb{A}^{n'}$ where $i_1,\ldots,i_{n-1} \in \{1,\ldots,5  \}$.
Then, the action of $\lambda_r(t)$ on $\mathbb{A}^{5}$ is given by
\begin{equation*}
\lambda_r(t) \cdot e_{i_1,\ldots,i_{n-1}}  = t^{2r(r^+_{i_1,\ldots,i_{n-1}}-r^-_{i_1,\ldots,i_{n-1}})} e_{i_1,\ldots,i_{n-1}}
\end{equation*}
where $i_1,\ldots,i_{n-1} \in \{1,\ldots,5  \}$ and $r^+_{i_1,\ldots,i_{n-1}}$ (resp. $r^-_{i_1,\ldots,i_{n-1}}$) is the number of $2$ (resp. $3$) in the index set $\{  i_1,\ldots,i_{n-1} \}$.
For $x^*\in \mathbb{A}^{n'}$ lying over $x \in \overline{\mathrm{Rep}_{n,\boldsymbol{k}}}$,
we write $x^*=\sum x^*_{i_1,\ldots,i_{n-1}} e_{i_1,\ldots,i_{n-1}}$, so that 
\begin{equation*}
\lambda_r(t) \cdot x^*  = \sum t^{ 2r r_{i_1,\ldots,i_{n-1}}}  x^*_{i_1,\ldots,i_{n-1}} e_{i_1,\ldots,i_{n-1}}
\end{equation*}
where $r_{i_1,\ldots,i_{n-1}}= r^+_{i_1,\ldots,i_{n-1}}-r^-_{i_1,\ldots,i_{n-1}}$,
and we put
\begin{equation}
\begin{aligned}
\mu^L(x,\lambda_r) := &\ \max \{ -r_{i_1,\ldots,i_{n-1}} \mid {i_1,\ldots,i_{n-1}} \mbox{ such that } x_{i_1,\ldots,i_{n-1}}^*\neq 0 \} \\
= &\ \sharp \left\{ i\ \middle|\ M_i=\begin{pmatrix}
a_i&b_i \\
c_i&d_i
\end{pmatrix}, c_i\neq0 \right\}-\sharp \left\{ i\ \middle|\ M_i=\begin{pmatrix}
0&1 \\
0&0
\end{pmatrix},e_i=0 \right\}.
\end{aligned}
\end{equation}
On the other hand, we have
\begin{align*}
&\sharp \left\{ i\ \middle|\ M_i=\begin{pmatrix}
a_i&b_i \\
c_i&d_i
\end{pmatrix}, c_i\neq0 \right\}  \\
&=(n-1)- \sharp \left\{ i\ \middle|\ M_i=\begin{pmatrix}
0&1 \\
0&0
\end{pmatrix},e_i=0 \right\}-
\sharp \left\{ i\ \middle|\ M_i=\begin{pmatrix}
a_i&b_i \\
0&d_i
\end{pmatrix}, e_i \neq 0 \right\}.
\end{align*}
Then, we have 
\begin{equation}
\begin{aligned}
&\mu^L(x,\lambda_r)\\
&\ = (n-1) - 2\sharp \left\{ i\ \middle|\ M_i=\begin{pmatrix}
0&1 \\
0&0
\end{pmatrix},e_i=0 \right\} 
- \sharp \left\{ i\ \middle|\ M_i=\begin{pmatrix}
a_i&b_i \\
0&d_i
\end{pmatrix}, e_i \neq 0 \right\}.
\end{aligned}
\end{equation}

By the Hilbert-Mumford criterion (see \cite[Theorem 2.1]{MFK} or \cite[Proposition 4.11]{Newstead}), the point $x$ is stable (resp. semi-stable) for this action 
if and only if $\mu^L(g\cdot x,\lambda_r) > 0$ (resp. $\ge0$) for every $g \in SL_2(\mathbb{C})$ and every 1-PS $\lambda_r$ of the form (\ref{1-PS def}).
If the point $x$ satisfies the condition $2m_1 < \sharp I^{\textit{nil}}$, then we have $\mu^L(g\cdot x,\lambda_r) > 0$ for any $g \in SL_2(\mathbb{C})$.
On the other hand, we consider the case where the point $x$ satisfies the condition $2m_1 \ge \sharp I^{\textit{nil}}$. 
There are at most two components of the decomposition (\ref{decomp index set}) of $I^{\textit{nil}}$ such that the cardinalities are $m_1$.
We denote by $I_{\textit{max}}^{\textit{nil}}$ the union of the components.
If the index set $I^{\textit{nil}} \setminus I_{max}^{\textit{nil}}$ is nonempty,
then we have $\mu^L(g\cdot x,\lambda_r) > 0$ for $g \in SL_2(\mathbb{C})$ 
such that $gM_ig^{-1}$ is the matrix (\ref{nilpotent matrix}) where $i \in I^{\textit{nil}} \setminus I_{max}^{\textit{nil}}$.
For $g \in SL_2(\mathbb{C})$ such that $gM_ig^{-1}$ is the matrix (\ref{nilpotent matrix}) where $i \in I_{max}^{\textit{nil}}$, we have
\begin{equation}\label{n-1-2m1+m2}
\mu^L(g\cdot x,\lambda_r) \ge (n-1) -( 2 m_1 + m_2).
\end{equation}
If the index $i \in I_{max}^{\textit{nil}}$ of the normalized matrix is a element of $I_{l}^{\textit{nil}}$ such that $\sharp I_{l}^{\textit{nil}}=m_1$ and $\sharp J_{l}=m_2$,
then the equality of (\ref{n-1-2m1+m2}) holds.
For the other matrix $g\in SL_2(\mathbb{C})$, we have $\mu^L(g\cdot x,\lambda_r) > 0$.
We have thus proved the proposition.
\end{proof}

We obtain a compactification of the character variety $\mathcal{R}_{n,\boldsymbol{k}}$.
\begin{dfn}
\begin{equation*}
\overline{\mathcal{R}_{n,\boldsymbol{k}}}:=\mathrm{Proj}\ H^0(\overline{\mathrm{Rep}_{n,\boldsymbol{k}}},L^{\otimes r})^{Ad(SL_2(\mathbb{C}))}.
\end{equation*}
\end{dfn}
The variety $\overline{\mathcal{R}_{n,\boldsymbol{k}}}$ is a projective algebraic variety.
This variety may have singular points on the boundary. 
Then, we should take a resolution of singular points of $\overline{\mathcal{R}_{n,\boldsymbol{k}}}$.  
In general, it is not easy to give a systematic resolution of singularities for any $n$.
On the following sections, we treat the cases for $n=4, 5$.
We will show that $\overline{\mathcal{R}_{n,\boldsymbol{k}}}$ is non-singular and the boundary divisor is a triangle of $\mathbb{P}^1$.
On Section \ref{section n=5}, we will treat the case for $n=5$.

\section{$n=4$}\label{section n=4}
Let
\begin{equation}
\left(
\arraycolsep5pt
\left(
\begin{array}{@{\,}cc|c@{\,}}
a_1&b_1&\\
c_1&d_1&\\
\hline 
&& e_1
\end{array}
\right),
\arraycolsep5pt
\left(
\begin{array}{@{\,}cc|c@{\,}}
a_2& b_2& \\
c_2& d_2&  \\
\hline 
&&e_2 
\end{array}
\right),
\arraycolsep5pt
\left(
\begin{array}{@{\,}cc|c@{\,}}
a_3 & b_3 &\\
c_3 & d_3 &\\
\hline 
&& e_3
\end{array}
\right)
\right)
\in \overline{\mathrm{Rep}_{4,\boldsymbol{k}}}.
\end{equation}
The compactification $\overline{\mathrm{Rep}_{4,\boldsymbol{k}}}$ is defined by the following equations in $\mathbb{P}^4\times \mathbb{P}^4 \times \mathbb{P}^4$
\begin{equation}
a_i+d_i=k_ie_i,\ (i=1,2,3),
\end{equation}
\begin{equation}
a_id_i-b_ic_i=e_i^2,\ (i=1,2,3),
\end{equation}
\begin{equation}\label{Tr equation n=4}
\mathrm{Tr}\left(\begin{pmatrix}
a_1&b_1\\
c_1&d_1
\end{pmatrix}
\begin{pmatrix}
a_2&b_2\\
c_2&d_2
\end{pmatrix}
\begin{pmatrix}
a_3&b_3\\
c_3&d_3
\end{pmatrix}\right)=k_4e_1e_2e_3.
\end{equation}

We analyze the stability.
If $e_i=0$ and $e_je_k\neq0\ ( j,k \in \{1,2,3\}\setminus \{i\} )$, then
$x$ is an unstable point if and only if $x$ is a point of the orbit of $(M_1,M_2,M_3)$ where
\begin{equation*}
M_i=
\arraycolsep5pt
\left(
\begin{array}{@{\,}cc|c@{\,}}
0&1&\\
0&0&\\
\hline 
&& 0
\end{array}
\right),
M_j=
\arraycolsep5pt
\left(
\begin{array}{@{\,}cc|c@{\,}}
a_j& b_j& \\
0& d_j&  \\
\hline 
&&e_j 
\end{array}
\right),
M_k=
\arraycolsep5pt
\left(
\begin{array}{@{\,}cc|c@{\,}}
a_k & b_k &\\
0 & d_k &\\
\hline 
&& e_k
\end{array}
\right).
\end{equation*}
If $e_i=0, e_j=0$, then $x$ is an unstable point if and only if $x$ is a point of the orbit of $(M_1,M_2,M_3)$ 
where two matrices in $M_1,M_2,M_3$ are 
\begin{equation*}
\arraycolsep5pt
\left(
\begin{array}{@{\,}cc|c@{\,}}
0&1&\\
0&0&\\
\hline 
&& 0
\end{array}
\right).
\end{equation*}

\begin{lem}
\textit{
The point $x \in \overline{\mathrm{Rep}_{4,\boldsymbol{k}}}$ is stable if and only if $x$ is semistable. 
}
\end{lem}

\begin{proof}
The point $x=(M_1,M_2,M_3)$ is not stable if only $x$ is normalized as follows.
\begin{equation*}
M_i=
\arraycolsep5pt
\left(
\begin{array}{@{\,}cc|c@{\,}}
0&1&\\
0&0&\\
\hline 
&& 0
\end{array}
\right),
M_j=
\arraycolsep5pt
\left(
\begin{array}{@{\,}cc|c@{\,}}
a_j& b_j& \\
c_j& d_j&  \\
\hline 
&&e_j 
\end{array}
\right),
M_k=
\arraycolsep5pt
\left(
\begin{array}{@{\,}cc|c@{\,}}
a_k & b_k &\\
0 & d_k &\\
\hline 
&& e_k
\end{array}
\right), \mbox{ where } c_j\neq0,
\end{equation*}
or
\begin{equation*}
M_i=
\arraycolsep5pt
\left(
\begin{array}{@{\,}cc|c@{\,}}
0&1&\\
0&0&\\
\hline 
&& 0
\end{array}
\right),
M_j=
\arraycolsep5pt
\left(
\begin{array}{@{\,}cc|c@{\,}}
0& 0& \\
1& 0&  \\
\hline 
&& 0
\end{array}
\right),
M_k=
\arraycolsep5pt
\left(
\begin{array}{@{\,}cc|c@{\,}}
a_k & b_k &\\
0 & d_k &\\
\hline 
&& e_k
\end{array}
\right).
\end{equation*}
However, the matrices are not satisfied the equation (\ref{Tr equation n=4}).
Then, there are no strictly semistable points.
\end{proof}

The following theorem shows that our compactification $\overline{\mathcal{R}_{4,\boldsymbol{k}}}$ of $\mathcal{R}_{4,\boldsymbol{k}}$ has the same configuration 
of the boundary divisor as the natural compactification of the Fricke-Klein cubic surface.

\begin{thm} 
\textit{
The boundary divisor of the compactification $\overline{\mathcal{R}_{4,\boldsymbol{k}}}$ is a triangle of three projective lines.
}
\end{thm}

\begin{proof}
We describe the boundary divisor explicitly.
Let $E_i$ be the image of the divisor $[e_i=0]$ on $\overline{\mathrm{Rep}_{4,\boldsymbol{k}}}$ 
by the quotient $\overline{\mathrm{Rep}_{4,\boldsymbol{k}}}\rightarrow \overline{\mathcal{R}_{4,\boldsymbol{k}}}$ ($i=1,2,3$).
First, we describe $[e_1=0]$.
We normalize $M_1$ by the $SL_2(\mathbb{C})$-conjugate action as the matrix (\ref{nilpotent matrix}).
The stabilizer subgroup of the matrix is 
$\left\{ 
\begin{pmatrix}
a&b\\
0&1/a
\end{pmatrix}
\right\}$.

By the stability, we obtain $c_{2}\neq 0$ and $ c_{3}\neq 0$. 
Since $c_{2}\neq 0$,
the matrices of the component $[e_1=0]$ are normalized by the action of this stabilizer subgroup:
\begin{equation}\label{normalization n=4}
\left(
\arraycolsep5pt
\left(
\begin{array}{@{\,}cc|c@{\,}}
0&1&\\
0&0&\\
\hline 
&& 0
\end{array}
\right),
\arraycolsep5pt
\left(
\begin{array}{@{\,}cc|c@{\,}}
0&-e_2^2&\\
c_{2}^2& k_2 c_{2} e_2 \\
\hline 
&& c_{2} e_2 
\end{array}
\right),
\arraycolsep5pt
\left(
\begin{array}{@{\,}cc|c@{\,}}
a_{3}&b_{3}&\\
c_{3}&d_{3}&\\
\hline 
&& e_3 
\end{array}
\right)
\right).
\end{equation}
The stabilizer subgroup of the normalized matrices is the torus group 
$\left\{ \begin{pmatrix}
a&0\\
0&a^{-1}
\end{pmatrix} \right\} $.

Before we consider the quotient by the torus group, 
we consider the normalized matrices (\ref{normalization n=4}).
The normalized matrices are defined by the following equations
\begin{equation}\label{n=4 D1}
\left\{
\begin{aligned}
&a_3+d_{3}=k_3 e_3, \\
&a_{3}d_{3}-b_{3}c_{3}=e_3^2,\\
&c_{2}a_{3}+k_2e_2c_{3}=0
\end{aligned}
\right.
\end{equation}
in the Zariski open set $c_2c_3 \neq 0$ of $\mathbb{P}^1 \times \mathbb{P}^4$.
By the equations $a_3+d_{3}=k_3 e_3$ and $a_{3}d_{3}-b_{3}c_{3}=e_3^2$, we obtain the equation
\begin{equation*}
(-a_3^2+ k_3a_{3}e_{3}-e_3^2)-b_{3}c_{3}=0.
\end{equation*}
Note that the equation define a hypersurface of degree 2 in $\mathbb{P}^3$, which is isomorphic to $\mathbb{P}^1\times \mathbb{P}^1$.
We put the coordinate $([S_3:T_3],[U_3:V_3]) \in \mathbb{P}^1\times \mathbb{P}^1$ such that
\begin{align*}
(S_3U_3)(T_3V_3)&=-a_3^2+k_3a_{3}e_{3}-e_3^2=-(a_3-\alpha_3^{+} e_3 )(a_3-\alpha_3^{-} e_3 )\\
S_3V_3&=b_{3}\\
T_3U_3&=c_{3}
\end{align*}
where $\alpha_i^{+},\alpha_i^{-}$ are eigenvalues of a matrix of the semisimple conjugacy class $\mathcal{C}_i$.
Then, we obtain the following transformation from $\mathbb{P}^1\times \mathbb{P}^1$ to the hypersurface of degree $2$ on $\mathbb{P}^3$:
\begin{equation}\label{isomorphism to P1P1}
\begin{aligned}
a_3&=\frac{\alpha_3^{-} S_3U_3+\alpha_3^{+} T_3V_3}{\alpha_3^{+}-\alpha_3^{-}}, &  b_3&=S_3V_3,\\
c_3&=T_3U_3, & d_3&=\frac{\alpha_3^{+} S_3U_3+\alpha_3^{-} T_3V_3}{\alpha_3^{+}-\alpha_3^{-}}, \\
e_3&=\frac{ S_3U_3+T_3V_3}{\alpha_3^{+}-\alpha_3^{-}}.
\end{aligned}
\end{equation}

Therefore, 
the normalized matrices are defined by 
\begin{equation}\label{boundary eq n=4}
c_{2}(\alpha_3^{-} S_3U_3 + \alpha_3^{+}T_3V_3)+k_2(\alpha_3^+-\alpha_3^-)e_2(T_3U_3)=0
\end{equation}
in the Zariski open set $c_2T_3U_3\neq 0$ of $\mathbb{P}^1 \times (\mathbb{P}^1\times \mathbb{P}^1)$.

We consider the quotient by the torus group. 
The torus action on $\mathbb{P}^1 \times (\mathbb{P}^1\times \mathbb{P}^1)$ is
\begin{equation*}
\begin{aligned}
\begin{pmatrix}
a&0\\
0&a^{-1}
\end{pmatrix}
\curvearrowright
&([c_2:e_2],[S_3:T_3],[U_3:V_3])\\
& \longmapsto ([a^{-1}c_2:ae_2],[aS_3:a^{-1}T_3],[a^{-1}U_3:aV_3]).
\end{aligned}
\end{equation*}
We consider the $SL_2(\mathbb{C})$-linearization $L=\bigotimes_{i=1}^3 p_i^*(\mathcal{O}_{\mathbb{P}^4}(1))$ on $\overline{\mathrm{Rep}_{4,\boldsymbol{k}}}$.
We take a pull-back of $L$ via the embedding
\begin{equation}
p_{e_1}\colon \mathbb{P}^1 \times (\mathbb{P}^1\times \mathbb{P}^1) \hookrightarrow \overline{\mathrm{Rep}_{4,\boldsymbol{k}}}
\end{equation}
 defined by the matrices (\ref{normalization n=4}) and the transform (\ref{isomorphism to P1P1}).
Let $L_{e_1}$ be the pull-back of $L$ on $\mathbb{P}^1 \times (\mathbb{P}^1\times \mathbb{P}^1)$.
We obtain the $T$-linearization on $L_{e_1}$ induced by the $SL_2(\mathbb{C})$-linearization $L$ on $\overline{\mathrm{Rep}_{4,\boldsymbol{k}}}$.
We consider the dual action on $H^0(\mathbb{P}^1 \times (\mathbb{P}^1\times \mathbb{P}^1),L_{e_1})$.
We have the following basis of the subspace consisting of invariant sections:
\begin{equation}
\begin{aligned}
&s_1=b_1\otimes c^2_2\otimes S_3U_3,& s_2=b_1 \otimes c^2_2\otimes T_3V_3,\\
& s_3=b_1\otimes c_2 e_2\otimes T_3U_3 &
\end{aligned}
\end{equation}
where $b_1\in H^0(\mathbb{P}^1 \times (\mathbb{P}^1\times \mathbb{P}^1), (p_{e_1}\circ p_1)^*(\mathcal{O}_{\mathbb{P}^4}(1)))$ corresponding to the $(1,2)$-entry of the matrix $M_1$.
The sections have the relation 
\begin{equation*}
\alpha_3^{-} s_1 + \alpha_3^+s_2+k_2(\alpha_3^+-\alpha_3^-)s_3=0
\end{equation*}
by the equation (\ref{boundary eq n=4}).
Therefore, we obtain $E_1\cong\mathbb{P}^1$.
In the same way, we also obtain $E_i\cong\mathbb{P}^1$ ($i=2,3$).

We show that $E_1$ and $E_2$ intersect at one point.
We substitute $e_2=0$ for (\ref{n=4 D1}). 
Then, we have the following equations 
\begin{equation*}
\left\{
\begin{aligned}
&a_{3}+d_{3}=k_3 e_3, \\
&a_{3}d_{3}-b_{3}c_{3}=e_3^2,\\
&a_{3}=0.
\end{aligned}
\right.
\end{equation*}
The locus defined by the equations above is a quadric curve in $\mathbb{P}^2$, which is isomorphic to $\mathbb{P}^1$.
There are two unstable points in the locus, $[b_{3}:c_{3}:e_3]=[0:1:0]$ and $[b_{3}:c_{3}:e_3]=[1:0:0]$.
The intersection is the quotient of $\mathbb{P}^1$ minus the two points by the torus action.
Then, the intersection is a point.
In the same way, the intersection of $E_2$ and $E_3$ (resp. $E_3$ and $E_1$) is a point.
\end{proof}

\section{$n=5$}\label{section n=5}
Let
\begin{equation*}
\left(
\arraycolsep5pt
\left(
\begin{array}{@{\,}cc|c@{\,}}
a_1&b_1&\\
c_1&d_1&\\
\hline 
&& e_1
\end{array}
\right),
\arraycolsep5pt
\left(
\begin{array}{@{\,}cc|c@{\,}}
a_2&b_2&\\
c_2&d_2&\\
\hline 
&& e_2
\end{array}
\right),
\arraycolsep5pt
\left(
\begin{array}{@{\,}cc|c@{\,}}
a_3& b_3& \\
c_3& d_3&  \\
\hline 
&&e_3
\end{array}
\right),
\arraycolsep5pt
\left(
\begin{array}{@{\,}cc|c@{\,}}
a_4 & b_4 &\\
c_4 & d_4 &\\
\hline 
&& e_4
\end{array}
\right)
\right)
\in \overline{\mathrm{Rep}_{5,\boldsymbol{k}}}.
\end{equation*}
The compactification $\overline{\mathrm{Rep}_{5,\boldsymbol{k}}}$ is defined by the following equations in $(\mathbb{P}^4)^4$
\begin{equation}
a_i+d_i=k_ie_i,\ (i=1,2,3,4),
\end{equation}
\begin{equation}
a_id_i-b_ic_i=e_i^2,\ (i=1,2,3,4),
\end{equation}
\begin{equation}\label{Tr condition of n=5}
\mathrm{Tr}\left(\begin{pmatrix}
a_1&b_1\\
c_1&d_1
\end{pmatrix}
\begin{pmatrix}
a_2&b_2\\
c_2&d_2
\end{pmatrix}
\begin{pmatrix}
a_3&b_3\\
c_3&d_3
\end{pmatrix}
\begin{pmatrix}
a_4&b_4\\
c_4&d_4
\end{pmatrix}\right)=k_5e_1e_2e_3e_4.
\end{equation}

We consider the stability condition.
\begin{lem}
\textit{
The closures of orbits of properly semistable points contain the point
\begin{equation}\label{properly semistable1}
s_1= \left(
\arraycolsep5pt
\left(
\begin{array}{@{\,}cc|c@{\,}}
0&1&\\
0&0&\\
\hline 
&& 0
\end{array}
\right),
\arraycolsep5pt
\left(
\begin{array}{@{\,}cc|c@{\,}}
0&1&\\
0&0&\\
\hline 
&& 0
\end{array}
\right),
\arraycolsep5pt
\left(
\begin{array}{@{\,}cc|c@{\,}}
0& 0& \\
1& 0&  \\
\hline 
&&0 
\end{array}
\right),
\arraycolsep5pt
\left(
\begin{array}{@{\,}cc|c@{\,}}
0 & 0 &\\
1 & 0 &\\
\hline 
&& 0
\end{array}
\right)
\right)
\end{equation}
or
\begin{equation}\label{properly semistable2}
s_2= \left(
\arraycolsep5pt
\left(
\begin{array}{@{\,}cc|c@{\,}}
0&1&\\
0&0&\\
\hline 
&& 0
\end{array}
\right),
\arraycolsep5pt
\left(
\begin{array}{@{\,}cc|c@{\,}}
0&0&\\
1&0&\\
\hline 
&& 0
\end{array}
\right),
\arraycolsep5pt
\left(
\begin{array}{@{\,}cc|c@{\,}}
0& 0& \\
1& 0&  \\
\hline 
&&0 
\end{array}
\right),
\arraycolsep5pt
\left(
\begin{array}{@{\,}cc|c@{\,}}
0 & 1 &\\
0 & 0 &\\
\hline 
&& 0
\end{array}
\right)
\right).
\end{equation}
Expect for the points of the orbits of $s_1$ and $s_2$, the stabilizer groups of every points are finite.
Each stabilizer group of the orbits of $s_1$ and $s_2$ is conjugate to the torus group $T=\left\{ \begin{pmatrix}
a&0\\
0&a^{-1}
\end{pmatrix} \right\} $.
}
\end{lem}

\begin{proof}
Let $x = (M_1,\ldots,M_4)$ be a property semistable point.
By Proposition \ref{stability condition}, we have $2m_1 +m_2=4$.
First, we consider the case where $m_1=1,m_2=2$.
We put
\begin{equation}
M_{i_1}= 
\arraycolsep5pt
\left(
\begin{array}{@{\,}cc|c@{\,}}
0&1&\\
0&0&\\
\hline 
&& 0
\end{array}
\right),
M_{i_2}=
\arraycolsep5pt
\left(
\begin{array}{@{\,}cc|c@{\,}}
*&*&\\
0&*&\\
\hline 
&& *
\end{array}
\right),
M_{i_3}=
\arraycolsep5pt
\left(
\begin{array}{@{\,}cc|c@{\,}}
*&*&\\
0&*&\\
\hline 
&& *
\end{array}
\right),
M_{i_4}=
\arraycolsep5pt
\left(
\begin{array}{@{\,}cc|c@{\,}}
*&*&\\
c_{i_4}&*&\\
\hline 
&& *
\end{array}
\right)
\end{equation}
where $\{ i_1,\ldots,i_4\}=\{1,\ldots,4\}$ and $c_{i_4}\neq 0$.
However, by the condition $c_{i_4}\neq 0$, the matrices do not satisfy the equation (\ref{Tr condition of n=5}).

Second, we consider the case where $m_1=2,m_2=0$.
We put
\begin{equation}\label{m_1=2 m_2=0}
M_{i_1}= 
\arraycolsep5pt
\left(
\begin{array}{@{\,}cc|c@{\,}}
0&1&\\
0&0&\\
\hline 
&& 0
\end{array}
\right),
M_{i_2}=
\arraycolsep5pt
\left(
\begin{array}{@{\,}cc|c@{\,}}
0&1&\\
0&0&\\
\hline 
&& 0
\end{array}
\right),
M_{i_3}=
\arraycolsep5pt
\left(
\begin{array}{@{\,}cc|c@{\,}}
*&*&\\
c_{i_3}&*&\\
\hline 
&& *
\end{array}
\right),
M_{i_4}=
\arraycolsep5pt
\left(
\begin{array}{@{\,}cc|c@{\,}}
*&*&\\
c_{i_4}&*&\\
\hline 
&& *
\end{array}
\right)
\end{equation}
where $\{ i_1,\ldots,i_4\}=\{1,\ldots,4\}$, $c_{i_3}\neq0$, and $c_{i_3}\neq 0$.
If $(i_1,i_2)=(1,3)$ or $(2,4)$, then the matrices do not satisfy the equation (\ref{Tr condition of n=5}).
Therefore, we consider the case where $(i_1,i_2)=(1,2)$, $(2,3)$, or $(3,4)$.
The 1-parameter subgroup (\ref{1-PS def}) acts on the matrices (\ref{m_1=2 m_2=0}).
For the matrices $M_{i_1}$ and $M_{i_2}$, the action is trivial.
The actions of the 1-parameter subgroup $\lambda_r(t)$ on $M_{i_3}$ and $M_{i_4}$ are
\begin{equation}
\begin{aligned}
\lambda_r(t) \cdot M_{i_3}&=
\arraycolsep5pt
\left(
\begin{array}{@{\,}cc|c@{\,}}
*&t^{2r} *&\\
t^{-2r}c_{i_3}&*&\\
\hline 
&& *
\end{array}
\right)
 &\lambda_r(t) \cdot M_{i_4}&=
\arraycolsep5pt
\left(
\begin{array}{@{\,}cc|c@{\,}}
*&t^{2r} *&\\
t^{-2r}c_{i_4}&*&\\
\hline 
&& *
\end{array}
\right)  \\
&=
\arraycolsep5pt
\left(
\begin{array}{@{\,}cc|c@{\,}}
t^{2r}*&t^{4r} *&\\
c_{i_3}&t^{2r}*&\\
\hline 
&& t^{2r}*
\end{array}
\right),
&
&=
\arraycolsep5pt
\left(
\begin{array}{@{\,}cc|c@{\,}}
t^{2r}*& t^{4r}*&\\
c_{i_4}&t^{2r}*&\\
\hline 
&& t^{2r}*
\end{array}
\right).
\end{aligned}
\end{equation}
Then, the limit $\lim_{t\rightarrow 0}\lambda_r \cdot M$ is the matrix (\ref{properly semistable1}) or (\ref{properly semistable2}).

Since the orbits of the points $s_1$ and $s_2$ are closed, the orbits have the maximum dimension of the stabilizer group, which is one dimension.
\end{proof}

We consider a resolution of properly semistable points. 
We take the blowing up along the orbits of $s_1$ and $s_2$:
\begin{equation}\label{first blowing up}
\widetilde{\overline{\mathrm{Rep}_{5,\boldsymbol{k}}}} \longrightarrow \overline{\mathrm{Rep}_{5,\boldsymbol{k}}}.
\end{equation}
The simultaneous action of $SL_2(\mathbb{C})$ on $\overline{\mathrm{Rep}_{5,\boldsymbol{k}}}$ 
induces an action on $\widetilde{\overline{\mathrm{Rep}_{5,\boldsymbol{k}}}}$.
By taking the blowing up (\ref{first blowing up}), the condition for stability and unstability is unchanging. 
On the other hand, the points of the exceptional divisors are stable points.
The points of orbits which are not closed are unstable points. 
Hence, there is no properly semistable point in $\widetilde{\overline{\mathrm{Rep}_{5,\boldsymbol{k}}}}$.  
(See \cite[Section 6]{Kir}).
We will show that the quotient of the blowing up is non-singular.
First, we describe the blowing up of $\overline{\mathrm{Rep}_{5,\boldsymbol{k}}}$ along the orbit of $s_1$.
Let $U_1$ and $U_2$ be the Zariski open sets $U_1=[b_1\neq0,b_2\neq0,c_3\neq0,c_4\neq0]$ and $U_2=[c_1\neq0,c_2\neq0,b_3\neq0,b_4\neq0]$ 
of $\overline{\mathrm{Rep}_{5,\boldsymbol{k}}} \subset \overline{\mathcal{C}}_1 \times \cdots \times \overline{\mathcal{C}}_4$.
Note that the orbit of $s_1$ is contained in $U_1\cup U_2$.
Since $\overline{\mathcal{C}}_i \cong \mathbb{P}^1 \times \mathbb{P}^1$ for $i=1,\ldots,4$ by the transformation (\ref{isomorphism to P1P1}),
we have 
\begin{equation}
U_i \subset \overline{\mathrm{Rep}_{5,\boldsymbol{k}}} \subset (\mathbb{P}^1 \times \mathbb{P}^1)^4 \mbox{ for $i=1,2$}.
\end{equation}
In the open sets $U_1$ and $U_2$, we put the following affine coordinates 
\begin{equation*}
([1:x_1],[y_1:1]),([1:x_2],[y_2:1]),([x_3:1],[1:y_3]),([x_4:1],[1:y_4]),
\end{equation*} 
and
\begin{equation*}
([z_1:1],[1:w_1]),([z_2:1],[1:w_2]),([1:z_3],[w_3:1]),([1:z_4],[w_4:1]),
\end{equation*} 
respectively.
In the open set $U_1$, the ideal of the orbit of $s_1$ is $(X_1,X_2,X_3,X_4,X_5)$ where 
\begin{align*}
X_0&:=e_1=\frac{ y_1+x_1}{\alpha_1^{+}-\alpha_1^{-}}, & X_1&:=e_2=\frac{ y_2+x_2}{\alpha_2^{+}-\alpha_2^{-}},\\
X_2&:=e_3=\frac{ y_3+x_3}{\alpha_3^{+}-\alpha_3^{-}}, & X_3&:=e_4=\frac{ y_4+x_4}{\alpha_4^{+}-\alpha_4^{-}},\\
X_4&:=x_1-x_2, & X_5&:=x_3-x_4.
\end{align*}
We can extend the torus action on $\overline{\mathrm{Rep}_{5,\boldsymbol{k}}}$ to the torus action on $\widetilde{\overline{\mathrm{Rep}_{5,\boldsymbol{k}}}}$ by
\begin{equation*}
\begin{pmatrix}
a&0\\
0&a^{-1}
\end{pmatrix}
\curvearrowright [X_0:X_1:X_2:X_3:X_4:X_5] \longmapsto [a^{-2}X_0:a^{-2}X_1:a^2X_2:a^2X_3:a^{-2}X_4:a^2X_5].
\end{equation*}
On the other hand, in the open set $U_2$, the ideal of the orbit of $s_1$ is $(Y_1,Y_2,Y_3,Y_4,Y_5)$ where 
\begin{align*}
Y_0&:=e_1=\frac{ z_1+w_1}{\alpha_1^{+}-\alpha_1^{-}}, & Y_1&:=e_2=\frac{ z_2+w_2}{\alpha_2^{+}-\alpha_2^{-}},\\
Y_2&:=e_3=\frac{ z_3+w_3}{\alpha_3^{+}-\alpha_3^{-}}, & Y_3&:=e_4=\frac{ z_4+w_4}{\alpha_4^{+}-\alpha_4^{-}},\\
Y_4&:=z_1-z_2, & Y_5&:=z_3-z_4.
\end{align*}
We can extend the torus action on $\overline{\mathrm{Rep}_{5,\boldsymbol{k}}}$ to the torus action on $\widetilde{\overline{\mathrm{Rep}_{5,\boldsymbol{k}}}}$ by
\begin{equation*}
\begin{pmatrix}
a&0\\
0&a^{-1}
\end{pmatrix}
\curvearrowright [Y_0:Y_1:Y_2:Y_3:Y_4:Y_5] \longmapsto [a^{2}Y_0:a^{2}Y_1:a^{-2}Y_2:a^{-2}Y_3:a^{2}Y_4:a^{-2}Y_5].
\end{equation*}
Hence, we have
\begin{equation*}
\widetilde{\overline{\mathrm{Rep}_{5,\boldsymbol{k}}}}_{s_1} 
\hookrightarrow (\overline{\mathrm{Rep}_{5,\boldsymbol{k}}} \setminus U_1\cup U_2) \cup (U_1 \times \mathbb{P}^5) \cup (U_2 \times \mathbb{P}^5)
\end{equation*}
where $\widetilde{\overline{\mathrm{Rep}_{5,\boldsymbol{k}}}}_{s_1}$ is the blowing up along the orbit of $s_1$.
The stabilizer group of any point in the exceptional divisor is
\begin{equation*}
\left\{
\begin{pmatrix}
1&0\\
0&1
\end{pmatrix}
,\begin{pmatrix}
-1&0\\
0&-1
\end{pmatrix}
\right\}.
\end{equation*}
This action is trivial. 
In the same way, we can describe the blowing up along the orbit of $s_2$.
 

\begin{thm}\label{thm}
\textit{
In the case of $n=5$, there exists a non-singular compactification of $\mathcal{R}_{5,\boldsymbol{k}}$ such that the boundary complex is a simplicial decomposition of sphere $S^3$.
}
\end{thm}

\begin{proof}
The outline of the proof is as follows.
We put 
\begin{equation*}
\widetilde{\overline{\mathcal{R}_{5,\boldsymbol{k}}}}:=  \widetilde{\overline{\mathrm{Rep}_{5,\boldsymbol{k}}}}/\!/SL_2(\mathbb{C}).
\end{equation*}
We have the six components of the boundary divisor of $\widetilde{\overline{\mathcal{R}_{5,\boldsymbol{k}}}}$:
the quotients of the proper transformations of the divisors $[e_1=0], [e_2=0], [e_3=0], [e_4=0]$ of $\overline{\mathrm{Rep}_{5,\boldsymbol{k}}}$ 
and the quotients of the exceptional divisors associated with blowing up along $s_1$ and $s_2$.
We denote by $E_1,E_2,E_3,E_4$ and $\textit{ex}_1,\textit{ex}_2$ each component. 
In Step 1, we describe the components $E_1,E_2,E_3$ and $E_4$ explicitly.
In Step 2, we describe the intersections $E_i\cap E_{j},\ i\neq j$.
In particular, the intersections $E_i\cap E_{i+1}, i=1,2,3,4$ (where $E_5$ implies $E_1$) are nonempty and irreducible.
On the other hand, the intersections $E_i\cap E_{i+2}, i=1,2$ are not irreducible.
The intersection $E_i\cap E_{i+2}$ consists of two components, denoted by $E_{i,i+2}^{+},E_{i,i+2}^{-}$.
Then, we take the blowing up along the components $E_{1,3}^{+},E_{1,3}^{-},E_{2,4}^{+},E_{2,4}^{-}$:
\begin{equation}\label{second blowing up}
\widetilde{X} \longrightarrow X:= \widetilde{\overline{\mathcal{R}_{5,\boldsymbol{k}}}}.
\end{equation}
We use the same notation $E_i$ which is the proper transform of $E_i$.
We denote by $\textit{ex}_{1,3}^{+},\textit{ex}_{1,3}^{-},\textit{ex}_{2,4}^{+},\textit{ex}_{2,4}^{-}$ the exceptional divisors associated with the blowing up (\ref{second blowing up}).
Consequently, the components of the boundary divisor of the compactification $\widetilde{X}$ of $\mathcal{R}_{5,\boldsymbol{k}}$ are
\begin{equation*}
E_1,E_2,E_3,E_4, \textit{ex}_1,\textit{ex}_2,\textit{ex}_{1,3}^{+},\textit{ex}_{1,3}^{-},\textit{ex}_{2,4}^{+},\textit{ex}_{2,4}^{-}.
\end{equation*}
Next, we see how $ex_i$ and the other components intersect.
In Step 3, we describe the $2$-dimensional simplices and the $3$-dimensional simplices.
Finally, we can describe the boundary complex of the boundary divisor of the compactification of the character variety. 

\textbf{Step 1}.
We describe the component $E_i$ (i.e. $[e_i=0]/\!/SL_2(\mathbb{C})$) explicitly.
We consider the case where $e_1=0$.
Let $D_i$ be the divisor $[e_i=0]$ on $\overline{\mathrm{Rep}_{5,\boldsymbol{k}}}$ for $i=1,\ldots, 4$.
Let $(M_1,\ldots,M_4)$ be a point on $D_1$.
We normalize the matrix $M_1$ by the $SL_2(\mathbb{C})$-conjugate action as the matrix (\ref{nilpotent matrix}).
The stabilizer subgroup of the matrix is the group of upper triangular matrices.
From the stability, we obtain $c_{2}\neq 0$, $ c_{3}\neq 0$ or $c_4\neq 0$. 
In the case of $c_2\neq 0$,
the matrices of the divisor $D_1$ are normalized by the action of this stabilizer subgroup:
\begin{equation}
\left(
\arraycolsep5pt
\left(
\begin{array}{@{\,}cc|c@{\,}}
0&1&\\
0&0&\\
\hline 
&& 0
\end{array}
\right),
\arraycolsep5pt
\left(
\begin{array}{@{\,}cc|c@{\,}}
0&-e_2^2&\\
c_{2}^2& k_2 c_{2} e_2 \\
\hline 
&& c_{2} e_2 
\end{array}
\right),
\arraycolsep5pt
\left(
\begin{array}{@{\,}cc|c@{\,}}
a_{3}&b_{3}&\\
c_{3}&d_{3}&\\
\hline 
&& e_3 
\end{array}
\right),
\arraycolsep5pt
\left(
\begin{array}{@{\,}cc|c@{\,}}
a_{4}&b_{4}&\\
c_{4}&d_{4}&\\
\hline 
&& e_4 
\end{array}
\right)
\right).
\end{equation}
Then, we have the locus defined by the following equations
\begin{equation}\label{n=5 D1}
\left\{
\begin{aligned}
&a_3+d_{3}=k_3 e_3 \\
&a_{3}d_{3}-b_{3}c_{3}=e_3^2\\
&a_4+d_{4}=k_4 e_4 \\
&a_{4}d_{4}-b_{4}c_{4}=e_4^2\\
&c_2a_3a_4+k_2e_2c_3a_4+c_2b_3c_4+k_2e_2d_3c_4=0
\end{aligned}
\right.
\end{equation}
in $(\mathbb{P}^1 \times (\mathbb{P}^4 \times \mathbb{P}^4)) \cap [c_2\neq 0]$.
The locus defined by $a_i+d_{i}=k_i e_i$ and $a_{i}d_{i}-b_{i}c_{i}=e_i^2$ in $\mathbb{P}^4$ is isomorphic to $\mathbb{P}^1\times \mathbb{P}^1$.
We put the coordinates $S_3,T_3,U_3,V_3$ and $S_4,T_4,U_4,V_4$ of $(\mathbb{P}^1\times \mathbb{P}^1)^2$ in the same way as in Section \ref{section n=4}.
Then, the locus of the normalized matrices is defined by the following equation
\begin{align*}
&c_2(\alpha_3^{-} S_3U_3+\alpha_3^{+} T_3V_3)(\alpha_4^{-} S_4U_4+\alpha_4^{+} T_4V_4)\\
&\quad +k_2e_2(\alpha_3^{+}-\alpha_3^{-})(T_3U_3)(\alpha_4^{-} S_4U_4+\alpha_4^{+} T_4V_4)\\
&\qquad +c_2(\alpha_3^{+}-\alpha_3^{-})(\alpha_4^{+}-\alpha_4^{-})(S_3V_3)(T_4U_4)\\
&\qquad \quad + k_2e_2(\alpha_4^{+}-\alpha_4^{-})(\alpha_3^{+} S_3U_3+\alpha_3^{-} T_3V_3)(T_4U_4)=0
\end{align*}
in $ ( \mathbb{P}^1)^5  \cap [c_2\neq 0]$.
Let $D_1^{c_2\neq 0}$ be the Zariski open set of the hypersrface in $ (\mathbb{P}^1 )^5$.
The torus action on $D_1^{c_2\neq 0}$ is the following action:
\begin{align*}
\begin{pmatrix}
a&0\\
0&a^{-1}
\end{pmatrix}
&\curvearrowright
([c_2:e_2],[S_3:T_3],[U_3:V_3],[S_4:T_4],[U_4:V_4])\\
& \mapsto ([a^{-1}c_2:ae_2],[aS_3:a^{-1}T_3],[a^{-1}U_3:aV_3],[aS_3:a^{-1}T_3],[a^{-1}U_3:aV_3]).
\end{align*}

In the same way as in the case $c_2\neq0$, we have the Zariski open sets of the hypersurfaces in $(\mathbb{P}^1)^5$ corresponding to $c_3\neq 0$ and $c_4\neq 0$, 
denoted by $D_1^{c_3\neq 0}$ and $D_1^{c_4\neq 0}$.
We glue $D_1^{c_2\neq 0}$, $D_1^{c_3\neq 0}$ and $D_1^{c_4\neq 0}$, denoted by $D_1'$.
We take the blowing up (\ref{first blowing up}).
Let $\widetilde{D}_1'$ be the proper transform of $D_1'$.
Then, the component of the boundary divisor $E_1$ is the quotient of $\widetilde{D}_1'$ by the torus action.
Similarly, we may describe the components $E_j\ (j=2,3,4)$.

\textbf{Step 2}.
We denote by $D_{i,j}$ the intersection of the divisors $[e_i=0]$ and $[e_j=0]$ on $\overline{\mathrm{Rep}_{5,\boldsymbol{k}}}$.
First, we consider the intersection of $E_1$ and $E_2$. 
We substitute $e_2=0$ for (\ref{n=5 D1}).
Then, we have the locus defined by the following equations 
\begin{equation*}
\left\{
\begin{aligned}
&a_3+d_{3}=k_3 e_3 \\
&a_{3}d_{3}-b_{3}c_{3}=e_3^2\\
&a_4+d_{4}=k_4 e_4 \\
&a_{4}d_{4}-b_{4}c_{4}=e_4^2\\
&a_3a_4+b_3c_4=0
\end{aligned}
\right.
\end{equation*}
in $(\mathbb{P}^1 \times (\mathbb{P}^4)^2 )\cap [c_2\neq0]$.
By the transform (\ref{isomorphism to P1P1}), we have the Zariski open set of the hypersurface in $(\mathbb{P}^1 )^5$, denoted by $D_{12}^{c_2\neq 0}$.
Next, we consider the case where $c_3\neq 0$.
In the same way as in the case where $c_2\neq0$, we have the locus defined by the following equations
\begin{equation*}
\left\{
\begin{aligned}
&a_2+d_{2}=0 \\
&a_{2}d_{2}-b_{2}c_{2}=0\\
&a_4+d_{4}=k_4 e_4 \\
&a_{4}d_{4}-b_{4}c_{4}=e_4^2\\
&d_2c_3^2a_4-c_2e_3^2c_4+k_3d_2c_3e_3c_4=0
\end{aligned}
\right.
\end{equation*}
in $(\mathbb{P}^1 \times (\mathbb{P}^4)^2 )\cap [c_3\neq0]$.
Since we may put $\begin{pmatrix}
a&b\\
c&d
\end{pmatrix}=\begin{pmatrix}
st&s^2\\
-t^2&-st
\end{pmatrix}$ where $a+d=0,ad-be=0$, we have
\begin{equation*}
\left\{
\begin{aligned}
&a_4+d_{4}=k_4 e_4 \\
&a_{4}d_{4}-b_{4}c_{4}=e_4^2\\
&t(sc_3^2a_4-te_3^2c_4+k_3sc_3e_3c_4)=0.
\end{aligned}
\right.
\end{equation*}
By the transform (\ref{isomorphism to P1P1}), we have the Zariski open set of the hypersurface in $(\mathbb{P}^1 )^5$, denoted by $D_{1,2}^{c_3\neq 0}$.
The locus $D_{1,2}^{c_3\neq 0}$ is not irreducible.
Now, we take the blowing up along the orbits of $s_1$ and $s_2$.
Let $\widetilde{D}_{1,2}^{c_3\neq 0}$ be the proper transform of $D_{1,2}^{c_3\neq 0}$.
Since an orbit of a point of
\begin{equation}\label{t=0 minus intersection}
[t=0] \setminus ([t=0] \cap[sc_3^2a_4-te_3^2c_4+k_3sc_3e_3c_4]) \subset D_{1,2}^{c_3\neq 0}
\end{equation}
are not closed, the points of the inverse image of (\ref{t=0 minus intersection}) on $\widetilde{D}_{1,2}^{c_3\neq 0}$ are unstable (see \cite[Lemma 6.6]{Kir}).
Then, the quotient of $\widetilde{D}_{1,2}^{c_3\neq 0}$ by the torus action is irreducible.
Next, we consider the case where $c_4\neq 0$.
In the same way as in the case where $c_3\neq0$, we have the Zariski open set of the hypersurface in $(\mathbb{P}^1 )^5$, denoted by $D_{1,2}^{c_4\neq 0}$.
We glue $D_{1,2}^{c_2\neq 0}$, $D_{1,2}^{c_3\neq 0}$ and $D_{1,2}^{c_4\neq 0}$, denoted by $D_{1,2}'$.
We take the proper transform of $D_{1,2}'$ of the blowing up along the orbits of $s_1$ and $s_2$, denoted by $\widetilde{D}_{1,2}'$.
Then, the intersection of $E_1$ and $E_2$ is the quotient of $\widetilde{D}_{1,2}'$ by the torus action, denoted by $E_{1,2}$.
The intersection $E_{1,2}$ is irreducible.


Second, we consider the intersection of $E_1$ and $E_3$. 
We substitute $e_3=0$ for (\ref{n=5 D1}).
Then, we have the locus defined by the following equations 
\begin{equation*}
\left\{
\begin{aligned}
&a_3+d_{3}=0 \\
&a_{3}d_{3}-b_{3}c_{3}=0\\
&a_4+d_{4}=k_4 e_4 \\
&a_{4}d_{4}-b_{4}c_{4}=e_4^2\\
&c_2a_3a_4+k_2e_2c_3a_4+c_2b_3c_4+k_2e_2d_3c_4=0
\end{aligned}
\right.
\end{equation*}
in $(\mathbb{P}^1 \times (\mathbb{P}^4)^2)\cap [c_2\neq0]$.
We put $a_3=st, b_3=s^2, c_3=-t^2, d_3=-st$. 
Then, we have the equations
\begin{equation}\label{not irr eqs}
\left\{
\begin{aligned}
&a_4+d_{4}=k_4 e_4 \\
&a_{4}d_{4}-b_{4}c_{4}=e_4^2\\
&(ta_4+sc_4)(c_2s-k_2e_2t)=0.
\end{aligned}
\right.
\end{equation}
We denote the two components $[ta_4+sc_4=0]$ and $[c_2s-k_2e_2t=0]$ by $D_{1,3}^{c_2\neq0,+}$ and $D_{1,3}^{c_2\neq0,-}$.

\begin{atte}
Any point $(M_1,M_2,M_3,M_4)$ on $D_{1,3}^{c_2\neq0,+}$ is conjugate to the following matrices
\begin{equation}\label{D13+}
\left(
\arraycolsep5pt
\left(
\begin{array}{@{\,}cc|c@{\,}}
0&1&\\
0&0&\\
\hline 
&& 0
\end{array}
\right),
\arraycolsep5pt
\left(
\begin{array}{@{\,}cc|c@{\,}}
a_2&b_2\\
c_2&d_2  \\
\hline 
&& e_2
\end{array}
\right),
\arraycolsep5pt
\left(
\begin{array}{@{\,}cc|c@{\,}}
0&0&\\
1&0&\\
\hline 
&& 0 
\end{array}
\right),
\arraycolsep5pt
\left(
\begin{array}{@{\,}cc|c@{\,}}
0&b_4&\\
c_4&d_4&\\
\hline 
&&e_4  
\end{array}
\right)
\right).
\end{equation}
In fact, 
we normalize the third matrix $M_3$ instead of $M_2$. 
Then, we have
\begin{equation*}
M_3=\arraycolsep5pt
\left(
\begin{array}{@{\,}cc|c@{\,}}
0&1&\\
0&0&\\
\hline 
&& 0 
\end{array}
\right) \mbox{ or }
\arraycolsep5pt
\left(
\begin{array}{@{\,}cc|c@{\,}}
0&0&\\
1&0&\\
\hline 
&& 0 
\end{array}
\right).
\end{equation*}
In the former case, by the stability, we have $c_4\neq 0$.
However, the matrices do not satisfy the condition (\ref{Tr condition of n=5}).
In the latter case, the equation $ta_4+sc_4=0$ implies that $a_4=0$.
On the other hand, any point on $D_{1,3}^{c_2\neq0,-}$ is conjugate to the following matrices
\begin{equation}\label{D13-}
\left(
\arraycolsep5pt
\left(
\begin{array}{@{\,}cc|c@{\,}}
0&1&\\
0&0&\\
\hline 
&& 0
\end{array}
\right),
\arraycolsep5pt
\left(
\begin{array}{@{\,}cc|c@{\,}}
a_2&b_2\\
c_2&0  \\
\hline 
&& e_2
\end{array}
\right),
\arraycolsep5pt
\left(
\begin{array}{@{\,}cc|c@{\,}}
0&0&\\
1&0&\\
\hline 
&& 0 
\end{array}
\right),
\arraycolsep5pt
\left(
\begin{array}{@{\,}cc|c@{\,}}
a_4&b_4&\\
c_4&d_4&\\
\hline 
&&e_4  
\end{array}
\right)
\right).
\end{equation}
\end{atte}

We consider the cases where $c_3\neq 0$ and $c_4\neq 0$.
In the same way as in the case where $c_2\neq 0$, we have the Zariski open sets
\begin{equation*}
D_{1,3}^{c_3\neq0,+},  D_{1,3}^{c_3\neq0,-}, D_{1,3}^{c_4\neq0,+},  D_{1,3}^{c_4\neq0,-}
\end{equation*}
of the hypersurfaces in $(\mathbb{P}^1 )^5$.
We glue $D_{1,3}^{c_2\neq 0,+}$, $D_{1,3}^{c_3\neq 0,+}$ and $D_{1,3}^{c_4\neq 0,+}$ 
(resp. $D_{1,3}^{c_2\neq 0,-}$, $D_{1,3}^{c_3\neq 0,-}$ and $D_{1,3}^{c_4\neq 0,-}$), denoted by ${}'D_{1,3}^+$ (resp. ${}'D_{1,3}^-$).
We take the blowing up (\ref{first blowing up}).
Let ${}'\widetilde{D}_{1,3}^+$ and ${}'\widetilde{D}_{1,3}^-$ be the proper transforms of ${}'D_{1,3}^+$ and ${}'D_{1,3}^-$, respectively.
Then, the intersections of $E_1$ and $E_3$ are the quotients of ${}'\widetilde{D}_{1,3}^+$ and ${}'\widetilde{D}_{1,3}^-$ by the torus action, denoted by $E_{1,3}^+$ and $E_{1,3}^-$.

We consider the intersections $E_2\cap E_3$, $E_3\cap E_4$ and $E_1\cap E_4$.
In the same way as in the case $E_1 \cap E_2$, the intersections are irreducible, denoted by $E_{2,3}$, $E_{3,4}$ and $E_{1,4}$.

We consider the intersection of $E_2$ and $E_4$.
In the same way as in the case $E_1 \cap E_3$, the intersection $E_2\cap E_4$ is not irreducible.
The intersection has two components, denoted by $E_{2,4}^+$ and $E_{2,4}^-$.
Here, the components $E_{2,4}^+$ and $E_{2,4}^-$ correspond respectively to the following matrices
\begin{equation*}
\left(
\arraycolsep5pt
\left(
\begin{array}{@{\,}cc|c@{\,}}
a_1&b_1\\
c_1&d_1  \\
\hline 
&& e_1
\end{array}
\right),
\arraycolsep5pt
\left(
\begin{array}{@{\,}cc|c@{\,}}
0&1&\\
0&0&\\
\hline 
&& 0 
\end{array}
\right),
\arraycolsep5pt
\left(
\begin{array}{@{\,}cc|c@{\,}}
a_3&b_3&\\
c_3&0&\\
\hline 
&&e_3  
\end{array}
\right),
\arraycolsep5pt
\left(
\begin{array}{@{\,}cc|c@{\,}}
0&0&\\
1&0&\\
\hline 
&& 0
\end{array}
\right)
\right)
\end{equation*}
and
\begin{equation*}
\left(
\arraycolsep5pt
\left(
\begin{array}{@{\,}cc|c@{\,}}
0&b_1\\
c_1&d_1  \\
\hline 
&& e_1
\end{array}
\right),
\arraycolsep5pt
\left(
\begin{array}{@{\,}cc|c@{\,}}
0&1&\\
0&0&\\
\hline 
&& 0 
\end{array}
\right),
\arraycolsep5pt
\left(
\begin{array}{@{\,}cc|c@{\,}}
a_3&b_3&\\
c_3&d_3&\\
\hline 
&&e_3 
\end{array}
\right),
\arraycolsep5pt
\left(
\begin{array}{@{\,}cc|c@{\,}}
0&0&\\
1&0&\\
\hline 
&& 0
\end{array}
\right)
\right).
\end{equation*}

Now, we take the blowing up along the components $E_{1,3}^{+},E_{1,3}^{-},E_{2,4}^{+},E_{2,4}^{-}$:
\begin{equation*}
\widetilde{X} \longrightarrow X:= \widetilde{\overline{\mathcal{R}_{5,\boldsymbol{k}}}}.
\end{equation*}
We use the same notation $E_i$ which is the proper transforms of $E_i$.
We denote by $\textit{ex}_{1,3}^{+},\textit{ex}_{1,3}^{-},\textit{ex}_{2,4}^{+},\textit{ex}_{2,4}^{-}$ the quotients of the exceptional divisors associated with this blowing up.
Consequently, we have the ten components of the boundary divisor of the compactification $\widetilde{X}$ of $\mathcal{R}_{5,\boldsymbol{k}}$
\begin{equation*}
E_1,E_2,E_3,E_4, \textit{ex}_1,\textit{ex}_2,\textit{ex}_{1,3}^{+},\textit{ex}_{1,3}^{-},\textit{ex}_{2,4}^{+},\textit{ex}_{2,4}^{-},
\end{equation*}
and we obtain that the intersections
\begin{equation*}
E_1\cap E_2,\quad E_2\cap E_3,\quad E_3\cap E_4,\quad E_4\cap E_1
\end{equation*}
and
\begin{equation*}
E_1\cap \textit{ex}_{1,3}^{\pm},\ E_3\cap \textit{ex}_{1,3}^{\pm},\ E_2\cap \textit{ex}_{2,4}^{\pm},\ E_4\cap \textit{ex}_{2,4}^{\pm}
\end{equation*}
are nonempty and irreducible.

We describe the intersections of the other pairs.
We consider the intersection of $\textit{ex}_{1,3}^{+}$ and $E_4$.
If we substitute $e_4=0$ for the matrix (\ref{D13+}),
then we have $d_4=0$.
Moreover, we have $b_4=0$ or $c_4=0$.
Then, we obtain that
\begin{equation*}
{}'D_{1,3}^+ \cap [e_4=0] = \{s_1,s_2\} \cup [\mbox{points whose orbits are not closed}].
\end{equation*}
By the blowing up along $s_1$ and $s_2$, we obtain that the intersection of $E_{1,3}^+$ and $E_4$ is empty (see \cite[Lemma 6.6]{Kir}).
Then, the intersection of $\textit{ex}_{1,3}^{+}$ and $E_4$ is empty.
In the same way as above, the intersections
\begin{equation*}
\textit{ex}_{1,3}^{-} \cap E_2 ,\quad \textit{ex}_{2,4}^{+} \cap E_3 ,\quad \textit{ex}_{2,4}^{-} \cap E_1 
\end{equation*}
are empty.
On the other hand, the intersections 
\begin{equation*}
\textit{ex}_{1,3}^{+} \cap E_2 ,\ \textit{ex}_{1,3}^{-} \cap E_4 ,\ \textit{ex}_{2,4}^{+} \cap E_1 ,\ \textit{ex}_{2,4}^{-} 
\cap E_3,\ \textit{ex}_{1,3}^{+} \cap  \textit{ex}_{1,3}^{-} ,\ \textit{ex}_{2,4}^{+} \cap  \textit{ex}_{2,4}^{-} 
\end{equation*}
are nonempty and irreducible. 
Next, we consider the intersections of the pairs containing $\textit{ex}_1$ or $\textit{ex}_2$.
The orbit of the point $s_1$ (resp. $s_2$) is contained in the components $D_1,\ldots,D_4$ and $D_{1,3}^{\pm},D_{2,4}^{\pm}$, respectively.
Here, $D_{1,3}^{\pm}$ and $D_{2,4}^{\pm}$ are the irreducible components of $D_{1,3}$ and $D_{2,4}$.
Then, the intersections $\textit{ex}_i \cap E_j$ and $\textit{ex}_i \cap \textit{ex}_{k,k+2}^{\pm}$ are nonempty and irreducible for $i=1,2$, $j=1,\ldots,4$ and $k=1,2$.
On the other hand, the orbits of the point $s_1$ and $s_2$ are not intersect.
Then, the intersection of $\textit{ex}_1$ and $\textit{ex}_2$ is empty.

\textbf{Step 3}.
We draw the vertexes and the $1$-dimensional simplices except $\textit{ex}_1$ and $\textit{ex}_2$.
Then, we obtain the following figure
\begin{figure}
\unitlength 0.1in
\begin{picture}( 18.3600, 15.8000)(  5.7000,-19.4000)
%
{\color[named]{Black}{%
\special{pn 8}%
\special{pa 1524 456}%
\special{pa 696 1184}%
\special{pa 1542 1908}%
\special{pa 2368 1178}%
\special{pa 1524 456}%
\special{pa 696 1184}%
\special{fp}%
}}%
%
{\color[named]{White}{%
\special{pn 0}%
\special{sh 0.300}%
\special{ia 1524 456 40 34  0.0000000  6.2831853}%
}}%
{\color[named]{Black}{%
\special{pn 8}%
\special{pn 8}%
\special{ar 1524 456 40 34  0.0000000  6.2831853}%
}}%
%
{\color[named]{White}{%
\special{pn 0}%
\special{sh 0.300}%
\special{ia 2368 1178 38 34  0.0000000  6.2831853}%
}}%
{\color[named]{Black}{%
\special{pn 8}%
\special{pn 8}%
\special{ar 2368 1178 38 34  0.0000000  6.2831853}%
}}%
%
{\color[named]{White}{%
\special{pn 0}%
\special{sh 0.300}%
\special{ia 700 1188 40 34  0.0000000  6.2831853}%
}}%
{\color[named]{Black}{%
\special{pn 8}%
\special{pn 8}%
\special{ar 700 1188 40 34  0.0000000  6.2831853}%
}}%
%
{\color[named]{White}{%
\special{pn 0}%
\special{sh 0.300}%
\special{ia 1542 1908 38 32  0.0000000  6.2831853}%
}}%
{\color[named]{Black}{%
\special{pn 8}%
\special{pn 8}%
\special{ar 1542 1908 38 32  0.0000000  6.2831853}%
}}%
%
{\color[named]{White}{%
\special{pn 0}%
\special{sh 0.300}%
\special{ia 1812 1190 38 34  0.0000000  6.2831853}%
}}%
{\color[named]{Black}{%
\special{pn 8}%
\special{pn 8}%
\special{ar 1812 1190 38 34  0.0000000  6.2831853}%
}}%
%
{\color[named]{White}{%
\special{pn 0}%
\special{sh 0.300}%
\special{ia 1528 946 38 34  0.0000000  6.2831853}%
}}%
{\color[named]{Black}{%
\special{pn 8}%
\special{pn 8}%
\special{ar 1528 946 38 34  0.0000000  6.2831853}%
}}%
%
{\color[named]{White}{%
\special{pn 0}%
\special{sh 0.300}%
\special{ia 1252 1188 38 34  0.0000000  6.2831853}%
}}%
{\color[named]{Black}{%
\special{pn 8}%
\special{pn 8}%
\special{ar 1252 1188 38 34  0.0000000  6.2831853}%
}}%
%
{\color[named]{White}{%
\special{pn 0}%
\special{sh 0.300}%
\special{ia 1532 1428 40 34  0.0000000  6.2831853}%
}}%
{\color[named]{Black}{%
\special{pn 8}%
\special{pn 8}%
\special{ar 1532 1428 40 34  0.0000000  6.2831853}%
}}%
%
{\color[named]{Black}{%
\special{pn 8}%
\special{pa 1524 456}%
\special{pa 1252 1188}%
\special{fp}%
\special{pa 1252 1188}%
\special{pa 1542 1908}%
\special{fp}%
\special{pa 1542 1908}%
\special{pa 1808 1186}%
\special{fp}%
\special{pa 1808 1186}%
\special{pa 1524 456}%
\special{fp}%
}}%
%
{\color[named]{Black}{%
\special{pn 8}%
\special{pa 696 1184}%
\special{pa 2368 1178}%
\special{fp}%
}}%
%
{\color[named]{Black}{%
\special{pn 8}%
\special{pa 1524 456}%
\special{pa 1542 1908}%
\special{dt 0.045}%
}}%
%
{\color[named]{Black}{%
\special{pn 8}%
\special{pa 696 1184}%
\special{pa 1532 1428}%
\special{dt 0.045}%
\special{pa 1536 1434}%
\special{pa 2368 1178}%
\special{dt 0.045}%
\special{pa 2368 1178}%
\special{pa 1528 946}%
\special{dt 0.045}%
\special{pa 1528 946}%
\special{pa 696 1184}%
\special{dt 0.045}%
\special{pa 696 1184}%
\special{pa 700 1188}%
\special{dt 0.045}%
}}%
\put(11.8000,-4.9000){\makebox(0,0)[lb]{$E_1$}}%
\put(5.7000,-14.2000){\makebox(0,0)[lb]{$E_2$}}%
\put(17.4000,-19.7000){\makebox(0,0)[lb]{$E_3$}}%
\put(23.8600,-10.7800){\makebox(0,0)[lb]{$E_4$}}%
\put(10.3000,-13.7000){\makebox(0,0)[lb]{$\textit{ex}_{1,3}^{+}$}}%
\put(18.0500,-11.2400){\makebox(0,0)[lb]{$\textit{ex}_{1,3}^{-}$}}%
\put(14.8000,-16.3000){\makebox(0,0)[lb]{$\textit{ex}_{2,4}^{-}$}}%
\put(14.7000,-8.8000){\makebox(0,0)[lb]{$\textit{ex}_{2,4}^{+}$}}%
\end{picture}%
\caption{}
\end{figure} 
We consider the following sphere
\begin{equation*}
\mathbb{R}^4 \supset S^3 = \{ (x,y,z,w) \in \mathbb{R}^4 \mid x^2+y^2+z^2+w^2=1 \}.
\end{equation*}
We arrange the vertexes except $\textit{ex}_1$ and $\textit{ex}_2$ on $S^2=S^3\cap [w=0]$ and
arrange the vertexes $\textit{ex}_1$ and $\textit{ex}_2$ at $(0,0,0,1)$ and $(0,0,0,-1)$ respectively.
We glue together the vertex $\textit{ex}_i\ (i=1,2)$ and each vertex on $S^2=S^3\cap [w=0]$.

Next, we describe the $2$-dimensional simplices.
First, we consider the intersections $E_1\cap E_2\cap \textit{ex}^+_{1,3}$ and $E_2\cap E_3\cap \textit{ex}^+_{1,3}$.
The intersection $E_1\cap E_2 \cap E_3 =E_{1,3}^+\cap E_2$ is nonempty and irreducible in $\widetilde{\overline{\mathcal{R}_{5,\boldsymbol{k}}}}$.
We take the blowing up along $E_{1,3}^+$.
Then, the intersections $E_1\cap E_2\cap \textit{ex}^+_{1,3}$ and $E_2\cap E_3\cap \textit{ex}^+_{1,3}$ are irreducible.
Second, we consider the intersections $E_1\cap \textit{ex}^+_{1,3} \cap \textit{ex}^-_{1,3}$ and $E_3\cap \textit{ex}^+_{1,3} \cap \textit{ex}^-_{1,3}$.
We substitute $d_2=0$ for the matrices (\ref{D13+}).
Then, we have that $D^{c_2\neq 0,+}_{1,3}\cap [d_2=0]$ is irreducible.
Therefore, the intersection $E^+_{1,3}\cap E_{1,3}^-$ is irreducible.
We take the blowing up along $E_{1,3}^+$.
Then, the intersections $E_1\cap \textit{ex}^+_{1,3} \cap \textit{ex}^-_{1,3}$ and $E_3\cap \textit{ex}^+_{1,3} \cap \textit{ex}^-_{1,3}$ are irreducible.
Then, we glue together the triangles 
\begin{equation*}
(E_1,E_2,\textit{ex}^+_{1,3}), (E_2,E_3,\textit{ex}^+_{1,3}),  (E_1,\textit{ex}^+_{1,3},\textit{ex}^-_{1,3} )\mbox{ and } (E_3,\textit{ex}^+_{1,3},\textit{ex}^-_{1,3} )
\end{equation*}
in the graph of Figure 3.
In the same way as above, we glue together each triangle.
Then, we obtain that the complex of Figure 3 is a simplicial decomposition of $S^2$.
Third, we consider the intersection of 3-tuple of components of the boundary divisor containing $\textit{ex}_1$ or $\textit{ex}_2$.
The divisors $\textit{ex}_1$ and $\textit{ex}_2$ are the exceptional divisors of the blowing up along the orbits of $s_1$ and $s_2$.
The orbits of $s_1$ and $s_2$ are contained in $D_i \cap D_{i+1}$ ($i=1,\ldots,4$), $D_{1,3}^{+}$, $D_{1,3}^{-}$, $D_{2,4}^{+}$ and $D_{2,4}^{-}$, respectively.
Then, the intersections $E_i\cap E_{i+1}\cap \textit{ex}_j$, $E^+_{k,k+2}\cap \textit{ex}_j$ and $E^-_{k,k+2}\cap \textit{ex}_j$ are nonempty and irreducible 
for $i=1,\ldots,4$, $j=1,2$, and $k=1,2$.
We take the blowing up along $E_{1,3}^+$ and $E_{1,3}^-$.
Then, we can glue together the $3$-tuples which have either $\textit{ex}_i$ or $\textit{ex}_i$ in the graph. 

Lastly, we describe the $3$-dimensional simplices.
We can glue together the $4$-tuples of components of the boundary divisor 
such that the $4$-tuples have either $\textit{ex}_i$ or $\textit{ex}_i$ and $3$-tuples expect $\textit{ex}_i$ or $\textit{ex}_i$ are glued together.
On the other hand, the intersections of the $4$-tuples which have the vertexes expect $\textit{ex}_i$ or $\textit{ex}_i$ are empty.
Then, we obtain that the boundary complex of the compactification $\widetilde{X}$ of $\mathcal{R}_{5,\boldsymbol{k}}$ is simplicial decomposition of $S^3$.
\end{proof}


\begin{thebibliography}{99}

\bibitem{abc} M. S. M. Benjamin,
{\it Compactifications of a representation variety},
 J. Group Theory {\bf 14} (2011), no. 6, 947-963.

\bibitem{CHM}
M. A. A. de Cataldo, T. Hausel, L. Migliorini, 
{\it Topology of Hitchin systems and Hodge theory of character varieties: the case $A_1$}, 
Ann. of Math. {\bf 175} (2012), no. 3, 1329-1407.


\bibitem{For}
E. Formanek, 
{\it The invariants of $n\times n$ matrices. Invariant theory}, 
18-43, Lecture Notes in Math., 1278, Springer, Berlin, 1987. 

\bibitem{FK}
R. Fricke, F. Klein,
{\it Vorlesungen uber die Theorie der automorphen Funktionen. Band 1: Die gruppentheoretischen Grundlagen. Band II: Die funktionentheoretischen Ausfuhrungen und die Andwendungen}.
Bibliotheca Mathematica Teubneriana, Bande 3, 4 Johnson Reprint Corp., New York; B. G. Teubner Verlagsgesellschaft, Stuttg art 1965 Band I: xiv+634 pp.; Band II: xiv+668 pp. 



\bibitem{HLR}
T. Hausel, E. Letellier, F. Rodriguez-Villegas, 
{\it Arithmetic harmonic analysis on character and quiver varieties},
Duke Math. Journal, vol. {\bf 160} (2011) 323-400.

\bibitem{HR}
T. Hausel, F. Rodriguez-Villegas,
{\it Mixed Hodge polynomials of character varieties},
with an appendix by N. M. Katz, Invent. Math. {\bf 174} (2008), 555-624.

\bibitem{IIS}
M. Inaba, K. Iwasaki, M.-H. Saito, 
{\it Moduli of stable parabolic connections, Riemann-
Hilbert correspondence and geometry of Painlev\'e equation of type VI. I },
Publ. Res. Inst. Math. Sci. (2006), no. {\bf 4}, 987-1089.

\bibitem{IS}
M. Inaba, M.-H. Saito, 
{\it Moduli of regular singular parabolic connections of spectral type on smooth projective curves}.
In preparation.

\bibitem{IIS2}
M. Inaba, K. Iwasaki, M.-H. Saito, 
{\it Moduli of stable parabolic connections, Riemann-Hilbert correspondence and geometry of Painlev\'e equation of type VI. II}.
Moduli spaces and arithmetic geometry, 387-432, Adv. Stud. Pure Math., {\bf 45}, Math. Soc. Japan, Tokyo, 2006.

\bibitem{Iwa} K. Iwasaki,
{\it An area-preserving action of the modular group on cubic Surfaces and the Painlev\'e VI Equations }, Comm. Math. Phys., {\bf 242} (2003), 185-219.


\bibitem{Jim}
M. Jimbo,
{\it Monodromy problem and the boundary condition for some Painlev\'e equations}.
Publ. Res. Inst. Math. Sci. {\bf 18} (1982), no. 3, 1137-1161. 


\bibitem{Kir}
F. Kirwan,
{\it Partial desingularisations of quotients of nonsingular varieties and their Betti numbers}, 
Ann. of Math. {\bf 122} (1985), no. 1, 41-85.

\bibitem{Lawton}
S. Lawton,
{\it Generators, relations and symmetries in pairs of $3\times3$ unimodular matrices }, J. Algebra {\bf 313} (2) (2007) 782-801.

\bibitem{MFK} D. Mumford, J. Fogarty, F. Kirwan, 
{\it Geometric invariant theory}, Third edition. Ergebnisse der Mathematik und ihrer Grenzgebiete (2) , 34. Springer-Verlag, Berlin, 1994.


\bibitem{Newstead}
P. E. Newstead, 
{\it Introduction to moduli problems and orbit spaces}.
Tata Institute of Fundamental Research Lectures on Mathematics and Physics, 51. Tata Institute of Fundamental Research, Bombay; by the Narosa Publishing House, New Delhi, 1978. vi+183 pp.


\bibitem{Pay}
S. Payne,
{\it Boundary complexes and weight filtrations}.
arXiv:1109.4286.

\bibitem{Pro} 
C. Procesi,
{\it The invariant theory of $n\times n$ matrices}, 
Advances in Math. {\bf 19} (1976), no. 3, 306-381.

\bibitem{Sim90}
C. Simpson,
{\it Harmonic bundles on noncompact curves}.
J. Amer. Math. Soc. {\bf 3} (1990), no. 3, 713--770. 

\bibitem{Sim}
C. Simpson,
{\it Moduli of representations of the fundamental group of a smooth projective variety. I},
Inst. Hautes \'Etudes Sci. Publ. Math. No. {\bf 79} (1994), 47-129.

\bibitem{Sim2}
C. Simpson,  
{\it Moduli of representations of the fundamental group of a smooth projective variety. II}, 
Inst. Hautes \'Etudes Sci. Publ. Math. No. {\bf 80} (1994), 5-79 (1995).

\bibitem{Sim3}
C. Simpson,
{\it Towards the boundary of the character variety},
Recent progress in the theory of Painlev\'e equations: algebraic, asymptotic and topological aspects, CNRS-JSPS, IRMA, Strasbourg, 4--8 november 2013.

\bibitem{Ste}
D. A. Stepanov, 
{\it A remark on the dual complex of a resolution of singularities}. 
 Uspekhi Mat. Nauk {\bf 61} (2006), no. 1(367), 185--186.

\bibitem{Thu}
A. Thuillier,
{\it G\'eom\'etrie toro\"\i dale et g\'eom\'etrie analytique non archim\'edienne. {A}pplication au type d'homotopie de certains sch\'emas formels}.
Manuscripta Math. {\bf 123} (2007), no. 4, 381--451.


\end{thebibliography}
\end{document}